\documentclass[11pt]{article}

\usepackage{amsfonts}
\usepackage{enumitem}
\usepackage{shuffle,yfonts,amsmath,amssymb,amsthm}
\usepackage{hyperref}
\hypersetup{pdfstartview={XYZ null null 1.00}}
\usepackage{pdfpages}

\newcommand\nc{\newcommand}
\nc{\ES}{\mathsf{ES}}
\nc{\MZV}{\mathsf{MZV}}
\nc{\CMZV}{\mathsf{CMZV}}
\nc{\MtV}{\mathsf{MtV}}
\nc{\AMtV}{\mathsf{AMtV}}
\nc{\MTV}{\mathsf{MTV}}
\nc{\MSV}{\mathsf{MSV}}
\nc{\MMV}{\mathsf{MMV}}
\nc{\MMVo}{\mathsf{MMVo}}
\nc{\MMVe}{\mathsf{MMVe}}
\nc{\AMMV}{\mathsf{AMMV}}
\nc{\sha}{\shuffle} 
\nc{\si}{\sigma}
\nc{\gd}{\delta}
\nc{\ola}{\overleftarrow}
\nc{\ora}{\overrightarrow}
\nc{\lra}{\longrightarrow}
\nc{\Lra}{\Longrightarrow}
\nc\Res{{\rm Res}}
\nc\setX{{\mathsf{X}}}
\nc\fA{{\mathfrak{A}}}
\nc\evaM{{\texttt{M}}}
\nc\evaML{{\text{\em{\texttt{M}}}}}
\nc\z{{\texttt{z}}}
\nc\ta{{\texttt{a}}}
\nc\ty{{\texttt{y}}}
\nc\tx{{\texttt{x}}}
\nc\td{{\texttt{d}}}
\nc\tz{{\texttt{z}}}
\nc\txp{{\tx_1}} 
\nc\txn{{\tx_{-1}}} 
\nc\neo{{1}}
\nc{\yi}{{1}}
\nc\one{{-1}}
\nc\om{{\omega}}
\nc\omn{\omega_{-1}}
\nc\omz{\omega_0}
\nc\omp{\omega_{1}}
\nc\eps{{\varepsilon}}
\nc{\bfp}{{\bf p}}
\nc{\bfq}{{\bf q}}
\nc{\bfu}{{\bf u}}
\nc{\bfv}{{\bf v}}
\nc{\bfw}{{\bf w}}
\nc{\bfy}{{\bf y}}
\nc{\bfga}{{\boldsymbol{\sl{\alpha}}}}
\nc{\bfe}{{\boldsymbol{\sl{e}}}}
\nc{\bfi}{{\boldsymbol{\sl{i}}}}
\nc{\bfj}{{\boldsymbol{\sl{j}}}}
\nc{\bfk}{{\boldsymbol{\sl{k}}}}
\nc{\bfl}{{\boldsymbol{\sl{l}}}}
\nc{\bfm}{{\boldsymbol{\sl{m}}}}
\nc{\bfn}{{\boldsymbol{\sl{n}}}}
\nc{\bfs}{{\boldsymbol{\sl{s}}}}
\nc{\bfr}{{\boldsymbol{\sl{r}}}}
\nc{\bft}{{\boldsymbol{\sl{t}}}}
\nc{\bfx}{{\boldsymbol{\sl{x}}}}
\nc{\bfz}{{\boldsymbol{\sl{z}}}}
\nc\bfmu{{\boldsymbol \mu}}
\nc\bfgl{{\boldsymbol \lambda}}
\nc\bfsi{{\boldsymbol \sigma}}
\nc\bfet{{\boldsymbol \eta}}
\nc\bfeta{{\boldsymbol \eta}}
\nc\bfeps{{\boldsymbol \varepsilon}}
\nc\bfone{{\bf 1}}

 \allowdisplaybreaks

\catcode`!=11
\let\!int\int \def\int{\displaystyle\!int}
\let\!lim\lim \def\lim{\displaystyle\!lim}
\let\!sum\sum \def\sum{\displaystyle\!sum}
\let\!sup\sup \def\sup{\displaystyle\!sup}
\let\!inf\inf \def\inf{\displaystyle\!inf}
\let\!cap\cap \def\cap{\displaystyle\!cap}
\let\!max\max \def\max{\displaystyle\!max}
\let\!min\min \def\min{\displaystyle\!min}
\let\!frac\frac \def\frac{\displaystyle\!frac}
\catcode`!=12

\let\oldsection\section
\renewcommand\section{\setcounter{equation}{0}\oldsection}

\allowdisplaybreaks

\DeclareMathOperator*{\dep}{dp}
\DeclareMathOperator{\Li}{Li}
\DeclareMathOperator{\ti}{ti}

\nc\UU{\mbox{\bfseries U}}
\nc\FF{\mbox{\bfseries \itshape F}}
\nc\h{\mbox{\bfseries \itshape h}}\nc\dd{\mbox{d}}
\nc\g{\mbox{\bfseries \itshape g}}
\nc\xx{\mbox{\bfseries \itshape x}}

\nc\gl{{\lambda}}
\nc\gL{{\Lambda}}
\nc\gD{{\Delta}}
\nc\ga{{\alpha}}
\nc\gb{{\beta}}
 \nc{\gam}{{\gamma}}
 \nc{\gG}{{\Gamma}}
 \nc{\vep}{{\varepsilon}}
 \nc{\gs}{{\sigma}}
 \nc{\gth}{{\theta}}
 \nc{\gS}{{\Sigma}}
 \nc{\gf}{{\phi}}
 \nc{\gk}{{\kappa}}
 \nc{\gm}{{\mu}}
 \nc{\gM}{{M}}
 \nc{\gz}{{\ze}}
 \nc{\tg}{{\tilde{g}}}
 \nc{\tgz}{{\tilde{\zeta}}}
 \nc{\gO}{{\Omega}}
 \nc{\sif}{{\mathcal S}}
 \nc{\gt}{{\tau}}
 \nc{\tlt}{{\tilde{t}}}
 \nc{\tgk}{{\tilde{\gk}}}
 \nc{\binn}{{\binom{2n}{n}}}

\def\N{\mathbb{N}}

\def\Q{\mathbb{Q}}
\def\CC{\mathbb{C}}

\def\ze{\zeta}

\def\xx{\left(\frac{1-x}{1+x} \right)}

\nc\divg{{\text{div}}}
\theoremstyle{plain}
\newtheorem{thm}{Theorem}[section]
\newtheorem{lem}[thm]{Lemma}
\newtheorem{cor}[thm]{Corollary}

\newtheorem{pro}[thm]{Proposition}

\theoremstyle{definition}

\newtheorem{re}[thm]{Remark}
\newtheorem{ex}[thm]{Example}

\nc{\myone}{{1}}
\nc{\myO}{{\mathsf O}}
\nc{\myL}{{\mathsf L}}

\let\Re\relax
\let\Im\relax
\DeclareMathOperator\Re{{{Re}}}
\DeclareMathOperator\Im{{{Im}}}

\begin{document}
\title{\bf Ap\'{e}ry-Type Series with Summation Indices of Mixed Parities and Colored Multiple Zeta Values, I} 
\author{
{Ce Xu${}^{a,}$\thanks{Email: cexu2020@ahnu.edu.cn, corresponding author, ORCID 0000-0002-0059-7420.}\ \ and Jianqiang Zhao${}^{b,}$\thanks{Email: zhaoj@ihes.fr, ORCID 0000-0003-1407-4230.}}\\[1mm]
\small a. School of Mathematics and Statistics, Anhui Normal University, Wuhu 241002, PRC\\
\small b. Department of Mathematics, The Bishop's School, La Jolla, CA 92037, USA}

\date{}
\maketitle

\noindent{\bf Abstract.} In this paper, we shall study A\'{e}ry-type series in which the central binomial coefficient appears as part of the summand. Let $b_n=4^n/\binom{2n}{n}$. Let $s_1,\dots,s_d$ be positive integers with $s_1\ge 2$. We consider
the series
\begin{align*}
\sum_{n_1>\cdots>n_d>0} \frac{b_{n_1}}{n_1^{s_1}\cdots n_d^{s_d}}
\end{align*}
and the variants with some or all indices $n_j$ replaced by $2n_j\pm 1$ and
some or all ``$>$'' replaced by ``$\ge$'', provided the series are defined.
We can also replace $b_{n_1}$ by its square in the above series when $s_1\ge 3$. The main result is that all such series
are $\Q$-linear combinations of the real and/or the imaginary parts of colored multiple zeta values of level 4, i.e., multiple polylogarithms evaluated at 4th roots of unity.



\section{Introduction}
In his celebrated proof of irrationality of $\gz(2)$ and $\gz(3)$ in 1979, Ap\'ery used crucially the following two identities
\begin{equation}\label{equ:AperyZ2Z3}
 \gz(2)=3\sum_{n\ge 1} \frac{1}{n^2\binn} \quad\text{and}\quad \gz(3)=\frac52\sum_{n\ge 1} \frac{(-1)^{n-1}}{n^3\binn}.
\end{equation}
Motivated by Ap\'ery's proof, Leshchiner \cite{Leshchiner} generalized these to higher weight Riemann zeta values
and some other analogs. However, no irrationality proof has been found so far for other Riemann zeta values at odd positive
integers greater than 4, although a lot of progress has been made (see e.g., \cite{LY2020,Rivoal2000,Zudilin2001}. In particular, in 2020, Lai and Yu \cite{LY2020} proved that for any small $\varepsilon>0$, the number of irrationals among the following odd zeta values: $\zeta(3),\zeta(5),\zeta(7),\ldots,\zeta(s)$, is at least $(c_0-\varepsilon)\sqrt{s/\log(s)}$, provided $s$ is a sufficiently large odd integer with respect to $\varepsilon$, with constant $c_0 = 1.192507\ldots$.

On the other hand, series generalizing those on the right-hand side of \eqref{equ:AperyZ2Z3},
including odd-indexed variations (see Remark~\ref{re-NeedOddVar}) have appeared in the calculations 
of the $\eps$-expansions of the Feynman diagrams in recent years (see, e.g., \cite{DavydychevDe2001,DavydychevDe2004,JegerlehnerKV2003}).

In the meantime, the frequent and sometimes unexpected appearance of colored multiple zeta values (see \eqref{equ:defnMPL})
in quite a few different branches of mathematics and physics has attracted the attention of many mathematicians and physicists alike.
These numbers, as vast generalizations of Riemann zeta values, are all conjectured to be not only irrational
but also transcendental. One naturally wonders if the multiple sums, which we call Ap\'ery-type series,
that generalize those in \eqref{equ:AperyZ2Z3} can be related to these numbers.
In a series of papers, we will answer some of these questions. As part I of this series, this paper concentrates on
Ap\'ery-type series such as those defined by \eqref{eqn:thm-Mixed-Parity-Any} and \eqref{eqn:thm-binnSquare-Any} in which
the central binomial coefficients appear only on the denominators.
We will show that a large class of these series can be expressed as $\Q$-linear combinations of the real and/or the imaginary parts
of the colored multiple zeta values of level 4, i.e., multiple polylogarithms evaluated at 4th roots of unity, see Thm.~\ref{thm-Mixed-Parity}. Some related results may be found in \cite{Au2020,KWY2007,Sun2015,X2020} and references therein.

\subsection{Notation.}
Let $\N$ be the set of positive integers and $\N_0:=\N\cup \{0\}$.
A finite sequence $\bfs:=(s_1,\ldots, s_d)\in\N^d$ is called a \emph{composition}. We define the weight and the depth of $\bfs$ by
\begin{equation*}
 |\bfs|:=s_1+\cdots+s_d,\quad\text{and}\quad \dep(\bfs):=d,
\end{equation*}
respectively. For any $N$th roots of unity $z_1,\dotsc,z_d$ the \emph{colored multiple zeta values} (CMZVs) of level $N$
are defined by
\begin{equation}\label{equ:defnMPL}
\Li_{\bfs}(\bfz):=\sum_{n_1>\cdots>n_d>0}
\frac{z_1^{n_1}\dots z_d^{n_d}}{n_1^{s_1} \dots n_d^{s_d}},
\end{equation}
which converge if $(s_1,z_1)\ne (1,1)$ (see \cite{Racinet2002} and \cite[Ch. 15]{Zhao2016}), in which case we call $(\bfs;\bfz)$ \emph{admissible}. The multiple zeta values are CMZVs of level 1, namely, $\ze(\bfs):=\Li_{\bfs}(1_d)$ where $1_d$ is the string of 1's with $d$ repetitions. Moreover, CMZVs can be expressed using Chen's iterated integrals
\begin{equation}\label{czeta-itIntExpression}
\Li_{\bfs}(\bfz)=\int_0^1 \ta^{s_1-1}\tx_{\xi_1}\cdots\ta^{s_d-1}\tx_{\xi_d},
\end{equation}
where $\xi_j:=\prod_{i=1}^j z_i^{-1}$, $\ta:=dt/t$ and $\tx_\xi:=dt/(\xi-t)$
for any $N$th roots of unity $\xi$, see \cite[Sec.~2.1]{Zhao2016} for a brief summary of this theory. The theory of iterated integrals was developed first by K.T. Chen in the 1960's. It has played important roles in the study of algebraic topology and algebraic geometry in the past half century. Its simplest form is
\begin{align*}
\int_0^1 f_1(t)dtf_{2}(t)dt\cdots f_p(t)dt
=&\, \int_0^1 f_1(t)dt\circ f_{2}(t)dt\circ \cdots \circ f_p(t)dt \\
:=&\, \int\limits_{1>t_1>\cdots>t_p>0}f_1(t_1)f_{2}(t_{2})\cdots f_p(t_p)dt_1dt_2\cdots dt_p.
\end{align*}
One can extend these to iterated integrals over any piecewise smooth path on the complex plane via pull-backs.
We refer the interested reader to Chen's original work \cite{KTChen1971,KTChen1977} for more details.

\subsection{Akhilesh's result.}
In \cite{Akhilesh1,Akhilesh} Akhilesh discovered some very important and surprising connections between
MZVs and the following Ap\'ery-type series (which he calls multiple Ap\'ery-like sums and which are normalized slightly differently here)
\begin{equation}\label{defn:gs}
\gs(\bfs;x):=\sum_{n_1> n_2>\cdots>n_d>0} {\binom{2n_1}{n_1}}^{-1} \frac{(2x)^{2n_1}}{(2n_1)^{s_1}\cdots (2n_d)^{s_d}}.
\end{equation}
His ingenious idea is to study the $n$-tails (and more generally, double tails) of such series.
We reformulate one of his key results as follows to make it more transparent. Set
\begin{align}\label{defn-g}
g_s(t)=
\left\{
 \begin{array}{ll}
 \tan t\, dt, \quad & \hbox{if $s=1$;} \\
 dt \circ (\cot t\, dt)^{s-2}\circ dt, \quad & \hbox{if $s\ge 2$,}
 \end{array}
\right.
\end{align}
and their non-trigonometric counterpart
\begin{equation}\label{defn-G}
 G_{s}(t)=
 \left\{
\begin{array}{ll}
 \om_{2} & \hbox{if $s=1$;} \\
 \om_{1}\om_0^{s-2}\om_{1} \qquad & \hbox{if $s\ge 2$,}
 \end{array}
 \right.
\end{equation}
where $\om$'s are defined by \eqref{defn:oms}. Further, we set $\binom{0}{0}=1$,
\begin{align*}
 b_n(x)=4^n{\binn}^{-1}x^{2n}\quad \text{and}\quad b_n= b_n(1)=4^n{\binn}^{-1}\quad \forall n\ge 0.
\end{align*}

\begin{thm} \label{thm-gs-Akhilesh} \emph{(\cite[Thm.~4]{Akhilesh})}
For all $n\in\N_0$, $\bfs=(s_1,\dots,s_d)\in\N^d$ we have
\begin{align*}
\gs(\bfs;\sin y)_n:=&\, \sum_{n_1 > \cdots > n_d>n} \frac{b_{n_1}(\sin y)}{(2n_1)^{s_1}\cdots (2n_d)^{s_d}}
=\frac{d}{dy} \int_0^y g_{s_1}\circ \cdots \circ g_{s_d}\circ b_{n}(\sin t) \,dt,
\end{align*}
where $y\in(-\pi/2,\pi/2)$ if $s_1=1$ and $y\in[-\pi/2,\pi/2]$ if $s_1\ge 2$. Using non-trigonometric 1-forms, we have
for all $x\in(-1,1)$
\begin{align*}
\gs(\bfs;x)_n:=&\, \sum_{n_1 > \cdots > n_d>n} \frac{b_{n_1}(x)}{(2n_1)^{s_1}\cdots (2n_d)^{s_d}}
=\sqrt{1-x^2}\frac{d}{dx} \int_0^x G_{s_1}\circ \cdots \circ G_{s_d}\circ b_{n}(t) \om_1.
\end{align*}
\end{thm}

We will again call the sum $|\bfs|:=s_1+\dots+s_d$ the weight of the series $\gs(\bfs;\sin y)$ and $d$ the depth.

\subsection{Main results.}
In Thm~\ref{thm-1stVariant} and Thm~\ref{thm-2ndVariant} we will show that the tails of the series
\begin{align*}
\gt^\star(\bfs;x)_n:= &\, \sum_{n_1\ge \cdots\ge n_d\ge n} {\binom{2n_1}{n_1}}^{-1} \frac{(2x)^{2n_1}}{(2n_1+1)^{s_1}\cdots (2n_d+1)^{s_d}},\\
\chi(\bfs;x)_n:=&\,\sum_{n_1> \cdots>n_d>n} {\binom{2n_1}{n_1}}^{-1} \frac{(2x)^{2n_1}}{(2n_1-1)^{s_1}\cdots (2n_d-1)^{s_d}}
\end{align*}
can be written as iterated integrals similar to the one in Thm~\ref{thm-gs-Akhilesh}. As corollaries, we show that the series
$\gs(\bfs;1)_0$, $\gt^\star(\bfs;1)_0$, $\chi(\bfs;1)_0$ and more general similar series defined
by \eqref{eqn:thm-Mixed-Parity-Any} and \eqref{eqn:thm-binnSquare-Any} with summation indices of any parity pattern
can be expressed as $\Q$-linear combinations of the real and/or the imaginary parts of CMZVs of level 4.

We also consider other Ap\'ery-type series similar to the above by using the
square of central binomial coefficients such as
\begin{equation}\label{defn:gsSq}
\begin{split}
\gs^{(2)}(\bfs):=&\, \sum_{n_1>\cdots>n_d>0} {\binom{2n_1}{n_1}}^{-2} \frac{16^{n_1}}{(2n_1)^{s_1}\cdots (2n_d)^{s_d}},\\
\gt^{\star,(2)}(\bfs):=&\, \sum_{n_1\ge\cdots\ge n_d\ge0} {\binom{2n_1}{n_1}}^{-2} \frac{16^{n_1}}{(2n_1+1)^{s_1}\cdots (2n_d+1)^{s_d}}
\end{split}
\end{equation}
and show that they also lie in the $\Q$-vector space of CMZVs of level 4.

\section{First variant with odd summation indices}
In this section, we consider a variation of the Ap\'ery-type series studied in \cite{Akhilesh} by restricting
the summation indices to odd numbers only and replacing strict inequalities among them by
non-strict ones. Concerning this we have the next well-known result. Define
\begin{alignat}{5}\label{defn:oms}
\om_0:=&\, \frac{dt}{t}, \quad &\,\om_1:=&\,\frac{dt}{\sqrt{1-t^2}}, \quad &\,\om_2:=&\,\frac{t\,dt}{1-t^2}, \\
\om_3:=&\, \frac{dt}{t\sqrt{1-t^2}}, \quad &\,\om_5:=&\,\frac{t\, dt}{\sqrt{1-t^2}}, \quad &\,\om_8:=&\,\frac{dt}{1-t^2}.
\end{alignat}

\begin{lem}
For all $d\in\N$ and $\bfs=(s_1,\dots,s_d)\in\N^d$, we have
\begin{align*}
\ti_\bfs(x)_n:=\sum_{n_1>\cdots>n_d\ge n} \frac{x^{2n_1+1}}{(2n_1+1)^{s_1}\cdots (2n_d+1)^{s_d}}
=\int_0^x \om_0^{s_1-1}\om_2\cdots \om_0^{s_{d-1}-1}\om_2\om_0^{s_d-1} (t^{2n} \om_8).
\end{align*}
\end{lem}
\begin{proof}
This follows easily by direct computation.
\end{proof}

For all $n\in\N_0$ and $\bfs=(s_1,\dots,s_d)\in\N^d$ we define
\begin{align*}
\gt(\bfs;x)_n:=&\, \sum_{n_1>\cdots>n_d\ge n} {\binom{2n_1}{n_1}}^{-1}
\frac{(2x)^{2n_1}}{(2n_1+1)^{s_1}\cdots (2n_d+1)^{s_d}},\\
\gt^\star(\bfs;x)_n:=&\, \sum_{n_1 \ge \cdots \ge n_d\ge n} {\binom{2n_1}{n_1}}^{-1}
\frac{(2x)^{2n_1}}{(2n_1+1)^{s_1}\cdots (2n_d+1)^{s_d}}.
\end{align*}

\begin{thm}
For all $d\in\N$ and $\bfs=(s_1,\dots,s_d)\in\N^d$, the tail
\begin{align*}
\gt(\bfs;1/2)_n
=&\, \int_0^{1/2} \frac{4t}{\sqrt{1-4t^2}}\om_0^{s_1-1}\om_2\cdots \om_0^{s_{d-1}-1}\om_2\om_0^{s_d-1} (t^{2n} \om_8).
\end{align*}
\end{thm}
\begin{proof}
Note that
\begin{align*}
\int_0^{1/2} \frac{4t}{\sqrt{1-4t^2}} t^{2n} dt
=&\, \int_0^{1/4} \frac{2t^{n}}{\sqrt{1-4t}} dt
=\frac{\Gamma( n+1)^2}{\Gamma(2n+2)}
={\binn}^{-1}\frac{1}{2n+1}.
\end{align*}
Thus
\begin{align*}
 &\,\sum_{n_1>\cdots>n_d\ge n} {\binom{2n_1}{n_1}}^{-1}\frac{1}{(2n_1+1)^{s_1}(2n_2+1)^{s_2}\cdots (2n_d+1)^{s_d}}\\
=&\,\int_0^{1/2} \frac{4t}{\sqrt{1-4t^2}}
\sum_{n_1>\cdots>n_d\ge n} \frac{ t^{2n_1} dt}{(2n_1+1)^{s_1-1}(2n_2+1)^{s_2}\cdots (2n_d+1)^{s_d}}\\
=&\,\int_0^{1/2} \frac{4t}{\sqrt{1-4t^2}} \bigg(\frac{d}{dt} \ti_\bfs(t)_n\bigg) dt\\
=&\,\int_0^{1/2} \frac{4t}{\sqrt{1-4t^2}} \om_0^{s_1-1}\om_2\cdots \om_0^{s_{d-1}-1}\om_2\om_0^{s_d-1} (t^{2n} \om_8),
\end{align*}
as desired.
\end{proof}

It turns out that the star version $\gt^\star$ behaves better. To study this,
we will extend Chen's iterated integrals by combining 1-forms and functions as follows.
For any $r\in\N$, 1-forms $f_1(t)\,dt,\dots,f_{r+1}(t)\,dt$ and functions $F_1(t),\dots,F_r(t)$, we define
recursively
\begin{align*}
&\, \int_0^1 \big( f_1(t)\,dt+F_1(t)\big)\circ \cdots\circ \big( f_r(t)\,dt+F_r(t)\big)\circ f_{r+1}(t)\,dt\\
:=&\, \int_0^1 \big( f_1(t)\,dt+F_1(t)\big)\circ \cdots\circ \big( f_{r-1}(t)\,dt+F_{r-1}(t)\big)\circ f_r(t)\,dt \circ f_{r+1}(t)\,dt\\
+&\, \int_0^1 \big( f_1(t)\,dt+F_1(t)\big)\circ \cdots\circ \big( f_{r-1}(t)\,dt+F_{r-1}(t)\big)\circ \big(F_r(t)f_{r+1}(t)\big)\,dt.
\end{align*}

We now set the 1-forms
\begin{align}\label{defn-h}
h_{s}(t)=
\left\{
 \begin{array}{ll}
 2\csc 2t\, dt, & \hbox{if $s=1$;} \\
 \csc t\,dt \circ (\cot t\, dt)^{s-2}\circ \csc t\, dt, \qquad& \hbox{if $s\ge 2$,}
 \end{array}
\right.
\end{align}
and their non-trigonometric counter part
\begin{align}\label{defn-H}
H_{s}(t)=
\left\{
 \begin{array}{ll}
\om_{20}:=\om_0+\om_2, \qquad & \hbox{if $s=1$;} \\
 \om_3 \om_0^{s-2} \om_3,& \hbox{if $s\ge 2$.}
 \end{array}
\right.
\end{align}

\begin{thm} \label{thm-1stVariant}
For all $n\in\N_0$, $\bfs=(s_1,\dots,s_d)\in\N^d$ we have the tail
\begin{align}\label{equ-thm-MtV-ItInt}
 \gt^\star(\bfs;\sin y)_n
 = & \, \sum_{n_1\ge \cdots \ge n_d\ge n}
\frac{b_{n_1}(\sin y)}{(2n_1+1)^{s_1}\cdots (2n_d+1)^{s_d}}
=\frac{d}{dy} \int_0^y h_{s_1}\circ \cdots h_{s_d}\circ b_{n}(\sin t) \,dt.
\end{align}
Hence
\begin{align*}
\gt^\star(\bfs;x)_n
=&\, \sqrt{1-x^2} \frac{d}{dx} \int_0^x H_{s_1}\circ \cdots \circ H_{s_d}\circ b_{n}(t) \om_1.
\end{align*}
\end{thm}

\begin{proof}
When $d=s_1=1$, the right-hand side of \eqref{equ-thm-MtV-ItInt} is equal to
\begin{align}\label{thm-1stVariant-iniStep}
 \frac{d}{dy} \int_0^y 2\csc 2t\, dt b_{n}(\sin t) \,dt
 =\sec y\csc y\int_0^y b_{n}(\sin t) \,dt.
\end{align}
Observe that
\begin{align*}
{\binom{2k+2}{k+1}}^{-1}\frac{1}{k+1}=\frac{(k+1)!^2}{(2k+2)(2k+1)(2k)!}\frac{1}{k+1}=\frac12{\binom{2k}{k}}^{-1}\frac{1}{2k+1}.
\end{align*}
By Thm.~\ref{thm-gs-Akhilesh} we get
\begin{align*}
\sum_{n_1\ge n} {\binom{2n_1+2}{n_1+1}}^{-1} \frac{(4\sin^2 y)^{n_1+1}}{n_1+1}
=2 \tan y \int_0^y b_{n}(\sin t) \, dt.
\end{align*}
Hence
\begin{align}\label{equ-V3Use}
\sum_{n_1\ge n} \frac{b_{n_1} \sin^{2n_1+1} y}{2n_1+1} =\sec y \int_0^y b_{n}(\sin t) \, dt,
\end{align}
which is exactly \eqref{thm-1stVariant-iniStep}. Thus the case $d=s_1=1$ of the theorem is proved.

Now, repeatedly multiplying \eqref{equ-V3Use} by $\cot y$ and integrating $s-1$ times, we get
\begin{align}\label{equ-V2Use}
\sum_{n_1\ge n} \frac{b_{n_1}\sin^{2n_1+1}y }{(2n_1+1)^s}
=\int_0^y (\cot t\, dt)^{s-2}\, (\csc t\, dt)\, ( b_{n}(\sin t)\, dt).
\end{align}
Replacing $n$ by $n_2$, multiplying by $1/(2n_2+1)$ and taking the sum $\sum_{n_1\ge n_2\ge n}$,
we get
\begin{align*}
\sum_{n_1\ge n_2\ge n} \frac{b_{n_1} \sin^{2n_1+1} y}{(2n_1+1)^s(2n_2+1)}
=&\, \int_0^y (\cot t\, dt)^{s-2} (\csc t\, dt) \sum_{n_2\ge n} \frac{b_{n_2}(\sin t)}{2n_2+1}\, dt\\
=&\, 2 \int_0^y (\cot t\, dt)^{s-2} \, (\csc t\, dt) \, (\csc 2t dt)\, ( b_{n}(\sin t)\, dt).
\end{align*}
Multiplying by $1/(2n_2+1)^{s_2}$ for $s_2\ge2$ we get
\begin{align*}
&\, \sum_{n_1\ge n_2\ge n} \frac{b_{n_1} \sin^{2n_1+1}y}{(2n_1+1)^{s_1}(2n_2+1)^{s_2}}\\
=&\, \int_0^y (\cot t\, dt)^{s-2} (\csc t\, dt) \sum_{n_2\ge n} \frac{b_{n_2}(\sin t)}{(2n_2+1)^{s_2}}\, dt\\
= &\, \int_0^y (\cot t\, dt)^{s_1-2} (\csc t\, dt) (\csc t\, dt) (\cot t\, dt)^{s_2-2} (\csc t\, dt) \, ( b_{n}(\sin t)\, dt).
\end{align*}
The theorem follows from doing these repeatedly and can be easily proved by induction. We leave the details to the interested reader.
\end{proof}

\begin{cor} For all admissible $\bfs=(s_1,\ldots,s_d)\in\N^d$ with $s_1\ge 2$, we have
\begin{align*}
t^\star(\bfs)_n:= \sum_{n_1\ge \cdots \ge n_d\ge n}
\frac{1}{(2n_1+1)^{s_1}\cdots (2n_d+1)^{s_d}}
=\frac2{\pi}\int_0^{\pi/2} h_{s_1} \circ \cdots \circ h_{s_d} \circ b_n(\sin t) \,dt.
\end{align*}
In particular,
\begin{align*}
t^\star(\bfs):= \sum_{n_1\ge \cdots \ge n_d\ge 0}
\frac{1}{(2n_1+1)^{s_1}\cdots (2n_d+1)^{s_d}}
=\frac2{\pi}\int_0^{\pi/2} h_{s_1} \circ \cdots \circ h_{s_d} \circ \,dt.
\end{align*}
\end{cor}
\begin{proof}
Integrating \eqref{equ-thm-MtV-ItInt} over $(0,\pi/2)$ and noticing the fact that
\begin{align*}
 \int_0^{\pi/2} (\sin t)^{2n}dt=\frac{\pi}{2 b_n},
\end{align*}
we obtain the corollary immediately.
\end{proof}

\begin{ex}
Let $\bfs=(2_d)$ be the string of 2's with $d$ repetitions. Then we see that
\begin{align*}
t^\star(2_d)=\frac2{\pi}\int_0^{\pi/2} \left(\frac{dt}{\sin t}\right)^{2d} dt
=\frac2{\pi} i \Big(\Li_{2d+1}(-i)-\Li_{2d+1}(i)\Big) =\frac4{\pi} \gb(2d+1),
\end{align*}
where $\beta$ is the Dirichlet beta function
\begin{align}\label{DirichletBeta}
 \gb(s)=\sum_{k\ge 0}\frac{(-1)^k}{(2k+1)^s}.
\end{align}
Moreover, we have
\begin{align*}
t^\star(2_a,3,2_b)=\frac2{\pi} \int_0^{\pi/2} (\csc t dt)^{2a+1}\circ \cot tdt \circ (\csc tdt)^{2b+1}\circ dt,
\end{align*}
where $a,b\in \N_0$.
\end{ex}

\begin{re} In \cite{Murakami2021}, T. Murakami first proved the analog of Zagier's 2-3-2 formula of MZVs for multiple $t$-values $t(2_a,3,2_b)$. Thereafter, several other proofs have appeared in the literature, see for example \cite{LLO2022}.
\end{re}

\begin{pro} \label{pro-gtIds}
For every positive integer $d$, we have
\begin{align*}
\gt^\star(1_d;\sin y)= &\, 2\csc 2y \Im \Li_d(i\tan y), \\
\gt^\star(2_d;\sin y)=&\, 2\csc y\Im \Li_{2d} \big(i\tan(y/2) \big).
\end{align*}
\end{pro}
\begin{proof}
By Thm.~\ref{thm-1stVariant} we obtain
\begin{align}
\gt^\star(1_d;\sin y)= &\, 2\csc 2y \int_0^y \left(\frac{2dt}{\sin 2t}\right)^{d-1} dt
= \frac{2\csc 2y }{(d-1)!} \int_0^y \left(\int_t^y\frac{2dx}{\sin 2x}\right)^{d-1} dt \notag \\
=&\, \frac{2\csc 2y }{(d-1)!} \int_0^y \log^{d-1}\left|\frac{\csc 2t+\cot 2t}{\csc 2y+\cot 2y}\right| dt \notag \\
=&\, \frac{2\csc 2y }{(d-1)!} \int_0^y \log^{d-1}\big|\tan y\cot t \big| \, dt. \label{chi0}
\end{align}
By routine differentiation and induction it can be proved easily that
\begin{align*}
\frac{1}{(d-1)!} \int_0^y \log^{d-1}\big|\tan y\cot t \big| \, dt= \Im \Li_d(i\tan y)
=\frac{i}{2} \Big(\Li_d(-i \tan y)-\Li_d(i\tan y) \Big)
\end{align*}
for all $d\ge 1$. To see that both sides $\to 0$ as $y\to 0$ we can use \eqref{chi0}.
Similarly,
\begin{align*}
\gt^\star(2_d;\sin y)=&\,\csc y \int_0^y \left(\frac{dt}{\sin t}\right)^{2d-1} dt
= \frac{\csc y}{(2d-1)!} \int_0^y \left(\int_t^y\frac{dx}{\sin x}\right)^{2d-1} dt \\
=&\, \frac{\csc y}{(2d-1)!} \int_0^y \log^{2d-1}\big|\tan y/2 \cot t/2\big| \, dt
=2\csc y\Im \Li_{2d} \big(i\tan(y/2) \big).
\end{align*}
This completes the proof of the proposition.
\end{proof}

\begin{ex}
Specializing at $y=\pi/4$ and $\pi/2$ in the two identities of Prop.~\ref{pro-gtIds} respectively, we see that
\begin{align*}
\sum_{n_1\ge \cdots \ge n_d\ge 0} {\binom{2n_1}{n_1}}^{-1}
\frac{2^{n_1}}{(2n_1+1)\cdots (2n_d+1)} =&\,2 \Im \Li_{d}(i)=2\gb(d),\\
\sum_{n_1\ge \cdots \ge n_d\ge 0} {\binom{2n_1}{n_1}}^{-1}
\frac{4^{n_1}}{(2n_1+1)^2\cdots (2n_d+1)^2} =&\, 2\Im \Li_{2d}(i)=2\gb(2d),
\end{align*}
where $\gb$ is the Dirichlet beta function defined by \eqref{DirichletBeta}.
\end{ex}

\section{Second variant with odd summation indices}
In this section, we consider another variation of the Ap\'ery-type series \eqref{defn:gs} by restricting
the summation indices to odd numbers only and \textbf{keeping} the strict inequalities among them. These
series do not behave as well as the first variant studied in the last section but are still of interest.

Define the 1-forms
\begin{align}\label{defn-gk}
\gk_{s}(t)=
\left\{
 \begin{array}{ll}
 \sin t\, dt\, \csc t\,dt+ \tan t\, dt,\quad & \hbox{if $s=1$;} \\
 \sin t\, dt (\cot t\,dt+1) (\cot t\,dt)^{s-2} \csc t\, dt, \phantom{\frac12} \quad & \hbox{if $s\ge 2$,}
 \end{array}
\right.
\end{align}
and their non-trigonometric counterpart
\begin{align}\label{defn-K}
K_{s}(t)=
\left\{
 \begin{array}{ll}
 \om_5\om_3+ \om_2,\quad & \hbox{if $s=1$;} \\
 \om_5 (\om_0+1) \om_0^{s-2} \om_3, \phantom{\frac12} \quad & \hbox{if $s\ge 2$.}
 \end{array}
\right.
\end{align}

\begin{thm} \label{thm-2ndVariant}
For all $n\in\N_0$ and $\bfs=(s_1,\dots,s_d)\in\N^d$ the tail
\begin{align*}
\chi(\bfs;\sin y)_n:=\sum_{n_1> \cdots>n_d>n} \frac{b_{n_1}(\sin y)}{(2n_1-1)^{s_1}\cdots (2n_d-1)^{s_d}}
= \frac{d}{d y} \int_0^y \gk_{s_1}\circ\cdots \circ \gk_{s_d}\circ b_{n}(\sin t) \,dt.
\end{align*}
In the above $y\in[-\pi/2,\pi/2]$ if $s_1>1$ and $y\in(-\pi/2,\pi/2)$ if $s_1=1$. Using
non-trigonometric 1-forms, for $x\in(-1,1)$ we have
\begin{align*}
\chi(\bfs;x)_n:=\sum_{n_1> \cdots>n_d>n} \frac{b_{n_1}(x)}{(2n_1-1)^{s_1}\cdots (2n_d-1)^{s_d}}
=\sqrt{1-x^2}\frac{d}{dx} \int_0^x K_{s_1}\circ\cdots \circ K_{s_d}\circ b_{n}(t) \om_1.
\end{align*}
\end{thm}

\begin{proof}
With $s=2$ the identity \eqref{equ-V2Use} yields that
\begin{align*}
\sum_{n_1> n} {\binom{2n_1-2}{n_1-1}}^{-1} \frac{4^{n_1-1} \sin^{2n_1-1} y}{(2n_1-1)^2}
=\int_0^y (\csc t\, dt) b_{n}(\sin t) \, dt.
\end{align*}
Noting that
\begin{equation*}
{\binom{2n_1-2}{n_1-1}}^{-1}={\binom{2n_1}{n_1}}^{-1}\frac{2(2n_1-1)}{n_1},
\end{equation*}
we have
\begin{align*}
\sum_{n_1> n} \frac{b_{n_1}(\sin y)}{n_1(2n_1-1)} =2\sin y \int_0^y (\csc t\, dt) b_{n}(\sin t)\, dt.
\end{align*}
Differentiating, we obtain
\begin{equation}\label{equ-V2}
\begin{split}
\sum_{n_1> n} \frac{b_{n_1} \sin^{2n_1-1} y\cos y}{2n_1-1}
=&\, \cos y \int_0^y (\csc t\, dt) b_{n}(\sin t)\, dt +\int_0^y b_{n}(\sin t)\, dt.
\end{split}
\end{equation}
Multiplying \eqref{equ-V2} by $\tan y$ we get
\begin{align}\label{chi-1}
\sum_{n_1> n} \frac{b_{n_1}(\sin y)}{2n_1-1}
=&\, \sin y \int_0^y (\csc t\, dt+\sec y) b_{n}(\sin t)\, dt.
\end{align}
Dividing \eqref{equ-V2} by $\sin y$ and integrating
\begin{align*}
\sum_{n_1> n} \frac{b_{n_1} \sin^{2n_1-1}y}{(2n_1-1)^2}
=&\, \int_0^y (\cot t\,dt) (\csc t\, dt) b_{n}(\sin t)\, dt
 +\int_0^y (\csc t\, dt) b_{n}(\sin t)\, dt.
\end{align*}
Repeatedly multiplying by $\cot y$ and integrating, we see that for all $s\ge 2$
\begin{align*}
\sum_{n_1> n} \frac{b_{n_1}\sin^{2n_1-1}y}{(2n_1-1)^s}
=&\, \int_0^y (\cot t\,dt)^{s-1} (\csc t\, dt) b_{n}(\sin t)\, dt
+ \int_0^y (\cot t\,dt)^{s-2} (\csc t\, dt) b_{n}(\sin t)\, dt.
\end{align*}
Hence if $s\ge 2$ then
\begin{align}\label{chi-s}
\sum_{n_1> n} \frac{b_{n_1}(\sin y)}{(2n_1-1)^s}
=&\, \sin y \int_0^y (\cot t\,dt+1) (\cot t\,dt)^{s-2} (\csc t\, dt) b_{n}(\sin t)\, dt.
\end{align}
The theorem now follows from repeatedly applying \eqref{chi-1} or \eqref{chi-s} at each depth.
This concludes the proof of the theorem.
\end{proof}

\section{Variant with summation indices of mixed parities}
By combining Thm.~\ref{thm-gs-Akhilesh}, Thm.~\ref{thm-1stVariant}, and Thm.~\ref{thm-2ndVariant} we obtain
the following result easily.

\begin{thm}\label{thm-Mixed-Parity}
Suppose $d\in \N$ and $\bfs=(s_1,\dots,s_d)\in\N^d$. Let $y\in(-\pi/2,\pi/2)$ if $s_1=1$ and $y\in[-\pi/2,\pi/2]$ if $s_1\ge 2$.
Set $\gl_{2n,s}(t)=g_s(t)$, $\gl_{2n+1,s}(t)=h_s(t)$ and $\gl_{2n-1,s}(t)=\gk_s(t)$ which are defined by \eqref{defn-g}, \eqref{defn-h} and \eqref{defn-gk}, respectively.
Then for any $l_1(n),\dots,l_d(n)=2n,2n\pm 1$ we have the tails
\begin{align}\label{eqn-Mixed-Parity-Trig}
& \sum_{n_1 \ \underset{1}{\succ} \ \dots\underset{d-1}{\succ} n_d\ \underset{d}{\succ} \ n}
\frac{b_{n_1}(\sin y)}{ l_1(n_1)^{s_1}\cdots l_d(n_d)^{s_d}}
 = \frac{d}{dy} \int_0^y \gl_{l_1,s_1} \circ \cdots \gl_{l_d,s_d} \circ b_{n}(\sin t)\,dt,
\end{align}
where ``$\underset{j}{\succ}$'' is ``$\ge$'' if $l_j(n)=2n+1$ and is ``$>$'' otherwise.
Using non-trigonometric 1-forms, we get for all $x\in(-1,1)$
\begin{align}\label{eqn-Mixed-Parity-NonTrig}
& \sum_{n_1 \ \underset{1}{\succ} \ \dots\underset{d-1}{\succ} n_d\ \underset{d}{\succ} \ n}
\frac{b_{n_1}(x)}{ l_1(n_1)^{s_1}\cdots l_d(n_d)^{s_d}}
 = \sqrt{1-x^2} \frac{d}{dx} \int_0^x \gL_{l_1,s_1} \circ \cdots \gL_{l_d,s_d} \circ b_{n}(t) \om_1,
\end{align}
where $\gL_{2n,s}(t)=G_s(t)$, $\gL_{2n+1,s}(t)=H_s(t)$ and $\gL_{2n-1,s}(t)=K_s(t)$ which are defined
by \eqref{defn-G}, \eqref{defn-H} and \eqref{defn-K}, respectively.
\end{thm}
\begin{proof} Define
\begin{align*}
f_1(t):=1,\quad f_2(t):= \frac{t}{\sqrt{1-t^2}},\quad f_3(t)=\frac1t, \quad f_{20}(t):= \frac{1}{t\sqrt{1-t^2}}, \quad f_5(t)=t.
\end{align*}
By Thm.~\ref{thm-gs-Akhilesh}, Thm.~\ref{thm-1stVariant}, and Thm.~\ref{thm-2ndVariant}
\begin{align}
 \sum_{n_1>n} \frac{b_{n_1}(x)}{2n_1} = &\, f_2(x)\int_0^x b_{n}(t) \om_1, \label{bn-it1} \\
 \sum_{n_1>n} \frac{b_{n_1}(x)}{(2n_1)^s}=&\, f_1(x) \int_0^x \om_0^{s-2} \om_1\, b_{n}(t)\om_1\quad\forall s\ge 2,\label{bn-it2} \\
 \sum_{n_1\ge n} \frac{b_{n_1}(x)}{2n_1+1}=&\,f_{20}(x)\int_0^x b_{n}(t) \om_1, \label{bn-it3} \\
 \sum_{n_1\ge n} \frac{b_{n_1}(x)}{(2n_1+1)^s}=&\,f_3(x)\int_0^x \om_0^{s-2} \om_3\, b_{n}(t) \om_1 \quad\forall s\ge 2, \label{bn-it4} \\
 \sum_{n_1>n} \frac{b_{n_1}(x)}{2n_1-1} = &\, f_5(x) \int_0^x \om_3\, b_{n}(t) \om_1 +f_2(x) \int_0^x b_{n}(t) \om_1, \label{bn-it5} \\
 \sum_{n_1>n} \frac{b_{n_1}(x)}{(2n_1-1)^{s}} = &\, f_5(x) \int_0^x (\om_0+1) \om_0^{s-2} \om_3\, b_{n}(t) \om_1\quad\forall s\ge 2. \label{bn-it6}
\end{align}
For convenience, we call the right-hand side of \eqref{bn-it1} and \eqref{bn-it2}
(resp. \eqref{bn-it3} and \eqref{bn-it4}, resp. \eqref{bn-it5} and \eqref{bn-it6})
a $\gs$-block (resp. $\gt^\star$-block, resp. $\chi$-block). In \eqref{eqn-Mixed-Parity-NonTrig},
each $s_j$ corresponds to (a variation of) such a block. Observing that $f_j(t)\om_1=\om_j$ we find that
after starting with a block in \eqref{bn-it1}-\eqref{bn-it6},
all the middle blocks should be modified as follows: (i) change $f_j(x)$ to $\om_j$, (ii) remove the
integral sign, and (iii) drop the 1-form $ b_{n}(t) \om_1$. Repeating this until the end block,
for which only operations (i) and (ii) are required.

This concludes the constructive proof of the theorem.
\end{proof}

\begin{re} \label{re-NeedOddVar} 
We note that Ap\'ery-type series with indices of mixed parity already appeared implicitly
in \cite[(1.1)]{DavydychevDe2004}. Indeed, using their notation one need to consider, for e.g., the following series:
\begin{align*}
\sum_{j=1}^\infty \frac{1}{\binom{2j}{j}} \frac{u^j}{j^c} S_a(2j-1)
=&\, 2^c\sum_{j=1}^\infty \frac{1}{\binom{2j}{j}} \frac{u^j}{(2j)^c}  \left(\sum_{k=0}^{j-1}\frac{1}{(2k+1)^a}+\sum_{k=1}^{j-1}\frac{1}{(2k)^a}\right) \\
=&\, 2^c \sum_{j>k\ge0} \frac{1}{\binom{2j}{j}} \frac{u^j}{(2j)^c(2k+1)^a} +2^c \sum_{j>k>0} \frac{1}{\binom{2j}{j}} \frac{u^j}{(2j)^c(2k)^a}.
\end{align*}
\end{re}

\begin{ex} By composing \eqref{bn-it2} and \eqref{bn-it3} we see that
\begin{align*}
 \sum_{n_1\ge n_2>0} \frac{b_{n_1}(x)}{(2n_1+1)(2n_2)^2}= \frac{1}{x\sqrt{1-x^2}}\int_0^x \om_1^3.
\end{align*}
Taking $x=1/2, \sqrt{3}/2$ we see immediately that
\begin{align*}
 \sum_{n_1\ge n_2>0}\binom{2n_1}{n_1}^{-1} \frac{1}{(2n_1+1)(2n_2)^2}= \frac{4}{\sqrt{3}}\frac{(\sin^{-1}(1/2))^3}{3!}=\frac{\pi^3}{4\cdot 81\sqrt{3}},\\
 \sum_{n_1\ge n_2>0}\binom{2n_1}{n_1}^{-1} \frac{3^n}{(2n_1+1)(2n_2)^2}= \frac{4}{\sqrt{3}}\frac{(\sin^{-1}(\sqrt{3}/2))^3}{3!}=\frac{2\pi^3}{81\sqrt{3}}.
\end{align*}
These are consistent with the first two identities at the beginning of \cite{Ablinger2015}. Many other evaluations in
the loc.\ cit.\ can be verified using similar ideas by repeatedly applying \eqref{bn-it1}--\eqref{bn-it6}. 
Taking $x=\sqrt{2}/2$ we also get
\begin{align*}
 \sum_{n_1\ge n_2>0}\binom{2n_1}{n_1}^{-1} \frac{2^n}{(2n_1+1)(2n_2)^2}= \frac{4}{\sqrt{3}}\frac{(\sin^{-1}(\sqrt{2}/2))^3}{3!}=\frac{\pi^3}{96\sqrt{3}}.
\end{align*}
\end{ex}

By specializing at $y=\pi/2$ (or taking limit as $x\to 1^-$) we obtain the following theorem,
which helps answer two questions at the end of \cite{XuZhao2021b} affirmatively in Cor.~\ref{cor-answerQuestions}.

\begin{thm}\label{thm-Mixed-Parity-CMZVs}
Suppose $d\in \N$, $\bfs=(s_1,\dots,s_d)\in\N^d$ and $s_1\ge 2$.
Set $\gd(l)=0$ if $l(n)=2n$ and $\gd(l)=1$ if $l(n)=2n\pm 1$.

\begin{enumerate}[label=\upshape{(\alph*)},leftmargin=1cm]
 \item \label{enu:thm-Mixed-Parity-2n2n+1}
 Suppose $l_1(n),\dots,l_d(n)=2n,2n+1$.
Then we have
\begin{align*}
\sum_{n_1 \ \underset{1}{\succ} \ \cdots\ \underset{d-1}{\succ} n_d\ \underset{d}{\succ} \ 0}
 \frac{b_{n_1}}{l_1(n_1)^{s_1}\cdots l_d(n_d)^{s_d}} \in i^{\gd(l_1)} \CMZV_{|\bfs|}^4,
\end{align*}
where ``$\underset{j}{\succ}$'' is ``$\ge$'' if $l_j(n)=2n+1$ and is ``$>$'' otherwise.

 \item \label{enu:thm-Mixed-Parity-2n-12n+1}
Suppose $l_1(n),\dots,l_d(n)=2n,2n\pm 1$. If for all $l_j(n)=2n-1$ ($j\ge2$) we have $l_{j-1}(n)\ne 2n$, then
we have
\begin{align}\label{equ-thm-Mixed-Parity1}
\sum_{n_1 \ \underset{1}{\succ} \ \cdots\ \underset{d-1}{\succ} n_d\ \underset{d}{\succ} \ 0}
 \frac{b_{n_1}}{l_1(n_1)^{s_1}\cdots l_d(n_d)^{s_d}} \in i^{\gd(l_1)} \Big( \CMZV_{|\bfs|}^4+\nu(l_1)\CMZV_{|\bfs|+1}^4\Big),
\end{align}
where $\nu(l)=1$ if $l(n)=2n-1$ and $\nu(l)=0$ otherwise. In particular, if $l_{j}(n)\ne 2n$ for all $j$ then
\eqref{equ-thm-Mixed-Parity1} holds.

 \item \label{enu:thm-Mixed-Parity-gen}
 More generally, $l_1(n),\dots,l_d(n)=2n,2n\pm 1$ then we have
\begin{align}\label{equ-thm-Mixed-Parity2}
\sum_{n_1 \ \underset{1}{\succ} \ \cdots\ \underset{d-1}{\succ} n_d\ \underset{d}{\succ} \ 0}
 \frac{b_{n_1}}{l_1(n_1)^{s_1}\cdots l_d(n_d)^{s_d}} \in \CMZV_{\le |\bfs|+\nu(l_1)}^4\otimes\Q[i].
\end{align}

 \item \label{enu:thm-Mixed-Parity-Any}
 Moreover, the claim in \eqref{equ-thm-Mixed-Parity2} still holds if one changes any of the strict inequalities
$n_j>n_{j+1}$ to $n_j\ge n_{j+1}$ and vice versa, provided the series is defined.
Here we set $n_{d+1}=0$. In particular, if $l_1(n),\dots,l_d(n)=2n,2n\pm 1$ then
\begin{align}\label{eqn:thm-Mixed-Parity-Any}
\sum_{n_1 \succ n_2 \succ\cdots\ \succ n_d\succ \ 0 }
 \frac{b_{n_1}}{l_1(n_1)^{s_1}\cdots l_d(n_d)^{s_d}} \in \CMZV_{\le |\bfs|+\nu(l_1)}^4\otimes\Q[i],
\end{align}
where ``$\succ$'' can be either ``$\ge$'' or ``$>$'', provided the series is defined.
\end{enumerate}
\end{thm}

\begin{re}
Let $q=\max\{j:l_j(n)\ne 2n+1\}$. The series is defined if and only if ``$\underset{q}{\succ}$'' is ``$>$''.
\end{re}

\begin{proof}
Put $\tx_\xi=dt/(\xi-t)$ for any $\xi\in\CC$ and $\td_{\xi,\xi'}=\tx_\xi-\tx_{\xi'}$.
First, we observe that under the change of variables
\begin{align}\label{equ-changVarSeq}
t\to \sin^{-1} t \quad\text{then}\quad t\to \frac{1-t^2}{1+t^2},
\end{align}
we have
\begin{alignat}{6}
\cot t\,dt \to &\, \om_{0}=\ta:=\frac{dt}{t}\to \ty,\quad &
 \csc t \,dt \to &\, \om_3:= \frac{dt}{t\sqrt{1-t^2}}\to \td_{-1,1}, \label{equ:changeVar1} \\
dt \to &\, \om_1:=\frac{dt}{\sqrt{1-t^2}} \to i \td_{-i,i}, \quad &
\sec t\csc t \, dt \to &\, \om_{20}:= \frac{dt}{t(1-t^2)}\to \ty+\tz, \label{1-formChangeVar2} \\
\tan t\,dt \to &\, \om_{2}:= \frac{t\,dt}{1-t^2} \to \tz, \quad &
\sec t\,dt \to &\, \om_8:= \frac{dt}{1-t^2} \to -\ta, \label{equ:changeVar3}
\end{alignat}
where $\ty=\tx_{-i}+\tx_{i}-\tx_{-1}-\tx_{1}$ and $\tz=-\ta-\tx_{-i}-\tx_{i}$.
Furthermore, we notice
\begin{align}
\sin t\,dt \to \om_5&\,=\frac{t\, dt}{\sqrt{1-t^2}} \to \frac{dt}{(i-t)^2}+\frac{dt}{(i+t)^2}.\label{equ:changeVar4}
\end{align}
Let
\begin{align*}
\myO:=\Q\left\langle \om_j : 0\le j\le 3 \text{ or } j=5 \right\rangle.
\end{align*}
By repeatedly using the six cases \eqref{bn-it1}--\eqref{bn-it6} we see that every sum in \eqref{equ-thm-Mixed-Parity2}
can be expressed as $\Q$-linear combinations of the following form
\begin{align}\label{afterIteration}
 \int_0^1 \ga_1 \dots \ga_m
\end{align}
with $m\le |\bfs|$ and $\ga_j\in\myO$.

\medskip
\ref{enu:thm-Mixed-Parity-2n2n+1} In this case $\om_5$ never appears and the weight in \eqref{bn-it1}--\eqref{bn-it4}
is always the same as the number of 1-forms appearing on the right-hand side so that there is
no weight drop. Thus we only need to consider the number $\om_1$'s appearing in \eqref{eqn-Mixed-Parity-NonTrig}.
From \eqref{equ:changeVar1}--\eqref{equ:changeVar3} we see that
only $\om_1$ produces $i$ after the change of variables $t\to (1-t^2)/(1+t^2)$.

Among the four cases \eqref{bn-it1}--\eqref{bn-it4} only iteration \eqref{bn-it2} affects the number of $\om_1$'s in \eqref{equ-thm-Mixed-Parity2}, by adding two $\om_1$'s. Then ending block $ b_n(t)\om_1= \om_1$ when $n=0$. So we
need to consider the starting 1-form inside the iterated integral of \eqref{eqn-Mixed-Parity-NonTrig},
which is chopped off after taking the derivative $d/dx$ and then multiplied by $\sqrt{1-x^2}$.
This 1-form is $\om_1$ if and only if it is a $\gs$-block with $s_1\ge 2$.
Hence, after taking $d/dx$ we find that on the right-hand side of \eqref{eqn-Mixed-Parity-NonTrig}
the total number of $\om_1$'s is even for a starting $\gs$-block (i.e., $l_1(n)=2n$) and the number is odd
for a starting $\gt^\star$-block (i.e., $l_1(n)=2n+1$). The claim of \ref{enu:thm-Mixed-Parity-2n2n+1} is thus proved.

\medskip
\ref{enu:thm-Mixed-Parity-2n-12n+1}
We first claim that we may reduce this case to the case where $l_j(n)=2n-1$ appears only when $j=1$.
We will prove this by induction on the depth. We have nothing to do when the depth is 1. In general,
we change the index $n_j\to n_j+1$ for all $j\ge2$ such that $l_j(n)=2n-1$.
Then we need to consider the following three possible blocks in front of $j$-th block.
Setting $k=n_{j-1},m=n_j, r=s_{j-1}, s=s_j$, we see that
\begin{align*}
& \sum_{k>m} \frac{1}{(2k-1)^{r} (2m-1)^{s} }\Longrightarrow
 \sum_{k>m} \frac{1}{(2k+1)^{r} (2m+1)^{s} }
 = \sum_{k\ge m} \frac{1}{(2k+1)^{r} (2m+1)^{s} }- \frac{1}{(2k+1)^{r+s} },\\
& \sum_{k\ge m} \frac{1}{(2k+1)^{r} (2m-1)^{s} }\Longrightarrow
 \sum_{k\ge m+1} \frac{1}{(2k+1)^{r} (2m+1)^{s} } \\
& \hskip5cm = \sum_{k\ge m} \frac{1}{(2k+1)^{r} (2m+1)^{s} }- \frac{1}{(2k+1)^{r+s}}.
\end{align*}
For the possible $\gt^\star$-block after the $j$-th block, setting $k=n_{j+1},m=n_j, r=s_{j+1}, s=s_j$ we have
\begin{align*}
& \sum_{k> m} \frac{1}{(2m-1)^{s} (2k+1)^{r} }\Longrightarrow
 \sum_{k+1> m} \frac{1}{ (2m+1)^{s} (2k+1)^{r}}
 = \sum_{k\ge m}\frac{1}{ (2m+1)^{s} (2k+1)^{r}}
\end{align*}
which looks in good shape. We also need to consider the possible $\gs$-blocks after the $j$-th block:
(setting $k=n_{j+1},m=n_j, r=s_{j+1}, s=s_j$)
\begin{align*}
& \sum_{k>m} \frac{1}{(2k-1)^{s} (2m)^{r} }\Longrightarrow
 \sum_{k+1>m} \frac{1}{(2k+1)^{s} (2m)^{r} }
 = \sum_{k\ge m} \frac{1}{(2k+1)^{s} (2m)^{r} },
\end{align*}
which looks in good shape, too.

To summarize, the above shows that if no $\gs$-$\chi$-block chain appears then we see that
no weight drops can occur in the decomposed sums.

Furthermore, from the above, we can assume the $\chi$-block appears only as the first block, if it ever does. By
the explicit iterated integral expressions of the $\gs$- and $\gt^\star$-blocks \eqref{bn-it1}--\eqref{bn-it4},
we see that if $\chi$-block does not appear then all the CMZVs involved are of the same weight. If a $\chi$-block
appears at the beginning then \eqref{bn-it6} shows that the weight can increase by one for some CMZVs
and the counting of $\om_1$ is the same as the case with a starting $\gt^\star$-block.
This completes the proof of \ref{enu:thm-Mixed-Parity-2n-12n+1}.

\ref{enu:thm-Mixed-Parity-gen}
As the proof of \ref{enu:thm-Mixed-Parity-2n-12n+1}, we first claim that we may reduce the general case
to the case where $l_j(n)=2n-1$ appears, if it ever does, then $j=1$.
We will prove this by induction on the depth. We have nothing to do when the depth is 1. In general,
we change the index $n_j\to n_j+1$ for all $j\ge2$ such that $l_j(n)=2n-1$.
Then we need to consider the following possible block $\gs$ in front of $j$-th block since
the other two possibilities have been already handled by case \ref{enu:thm-Mixed-Parity-2n-12n+1}:
(setting $k=n_{j-1},m=n_j, r=s_{j-1}, s=s_j$)
\begin{align*}
& \sum_{k>m} \frac{1}{(2k)^{r} (2m-1)^{s} }\Longrightarrow
 \sum_{k>m+1} \frac{1}{(2k)^{r} (2m+1)^{s} }
 = \sum_{k>m} \frac{1}{(2k)^{r} (2m+1)^{s} }- \frac{1}{(2k)^{r} (2k-1)^{s} }.
\end{align*}
Thus by partial fractions we may decompose the second term above as pure powers of either $2k$ or $2k-1$,
thus reducing the depth by 1. We point out that this is also the reason why
the weight may drop due to the partial fractions when $\gs$-$\chi$-block chain appears.

Thus we will assume the $\chi$-block appears only as the first block. Then this case is proved again
by the explicit formula \eqref{bn-it6}.

\medskip
\ref{enu:thm-Mixed-Parity-Any} Observe that for any fixed $n$,
\begin{align*}
\sum_{m\ge n}\frac{1}{m^s(2n\pm 1)^t}=&\, \frac{1}{n^s(2n\pm 1)^t}+\sum_{m> n}\frac{1}{m^s(2n\pm 1)^t},\\
\sum_{m> n}\frac{1}{(2m+1)^t n^s}=&\, -\frac{1}{n^s(2n+1)^t}+\sum_{m\ge n}\frac{1}{(2m+1)^t n^s},\\
\sum_{m> n}\frac{1}{(2m+1)^t (2n-1)^s}=&\, -\frac{1}{(2n+1)^t(2n-1)^s}+\sum_{m\ge n}\frac{1}{(2m+1)^t (2n-1)^s},\\
\sum_{m\ge n}\frac{1}{(2m-1)^t n^s}=&\, -\frac{1}{(2n-1)^t n^s}+\sum_{m> n}\frac{1}{(2m+1)^t n^s},\\
\sum_{m\ge n}\frac{1}{(2m-1)^t (2n+1)^s}=&\, \frac{1}{(2n-1)^t(2n+1)^s}+\sum_{m> n}\frac{1}{(2m+1)^t(2n+1)^s}.
\end{align*}
By partial fraction decomposition we can reduce all the first terms on the right-hand side of the above to a
$\Q$-linear combination of single powers of either $1/n$ or $1/(2n\pm 1)$. The only complication
is that when $j=1$ non-admissible terms may appear, which requires us to use the shuffle regularization.
This is similar to the proof of \cite[Thm.~9.6]{XuZhao2021b} using
\cite[Lemma.~9.4 and Lemma 9.5]{XuZhao2021b}. We leave the details to the interested reader
who may also refer to Example~\ref{exampleS} (especially the computation for $S_2$ on page \pageref{S2})
for the detailed steps to carry this out explicitly.
\end{proof}

\begin{cor}\label{cor:algPts}
Let $d\in \N$ and $\bfs=(s_1,\dots,s_d)\in\N^d$. Let $\bfl=(l_1,\dots,l_d)$ where $l_j(n)=2n$ or $2n\pm 1$ for all $j=1,\dots,d$.
Let ``$\succ$'' denote either ``$\ge$'' or ``$>$''. Then for every real algebraic $x$ such that $0<x\le 1$ the value
\begin{align*}
\Xi(\bfs,\bfl;x):=\sum_{n_1 \succ n_2 \succ\cdots\ \succ n_d\succ \ 0 }
 \frac{b_{n_1}(x)}{l_1(n_1)^{s_1}\cdots l_d(n_d)^{s_d}},
\end{align*}
if it exists, can be expressed as a $\Q[i,x,\sqrt{1-x^2}]$-linear combination of the multiple polylogarithms
evaluated at algebraic points.
\end{cor}
\begin{proof}  
By the proof of Thm.~\ref{thm-Mixed-Parity-CMZVs}, up to factors of $x$ and $\sqrt{1-x^2}$ in front
(which can be seen more easily from \eqref{bn-it1}--\eqref{bn-it6}), $\Xi(\bfs,\bfl;x)$ can be expressed as
a $\Q$-linear combinations of
\begin{align}\label{equ-oms-1forms}
 \int_0^x \Big[\om_j: j=0,1,2,3\Big]_\ell
\end{align}
where $\Big[\om_j: j=0,1,2,3\Big]_{|\bfs|}$ is an iteration of 1-forms of length $\ell\le |\bfs|+1$.
Here the 1-form $\om_5$ is not needed since the proof of Thm.~\ref{thm-Mixed-Parity-CMZVs} shows that
if $l_j(n)=2n-1$ then we may assume $j=1$ (i.e., $\chi$-block only appears at the beginning).
Therefore, after applying the change of variables $t\to \frac{1-t^2}{1+t^2}$, 
by \eqref{equ:changeVar1}--\eqref{equ:changeVar3} we see that \eqref{equ-oms-1forms}
is transformed to a $\Q[i]$-linear combination of iterated integrals of the form
\begin{align*}
 \int_{\gl(x)}^1 \ga_1\dots \ga_\ell, .
\end{align*}
where $\gl(x)=\sqrt{\frac{1-x}{1+x}},\ \ga_j\in \{\tx_0, \tx_\mu: \mu^8=1\}. $
Note $\gl(x)\to x$ under the change of variables $t\to \frac{1-t^2}{1+t^2}$. If $x\ne 1$, to convert this to multiple polylogs we
generally need to use the regularization process. Thus, for an arbitrarily small $\eps>0$ we write
\begin{align*}
 \int_{\gl(x)}^1 \ga_1\dots \ga_\ell=&\, \sum_{j=0}^\ell \int_{\gl(x)}^\eps \ga_1\ga_2\dots \ga_j \int_\eps^1 \ga_{j+1}\dots \ga_\ell\\
 =&\, \sum_{j=0}^\ell (-1)^j \int_\eps^{\gl(x)} \ga_j\dots\ga_2 \ga_1 \int_\eps^1 \ga_{j+1}\dots \ga_\ell\\
 =&\, \sum_{j=0}^\ell (-1)^j \int_{\eps/\gl(x)}^1 \ga'_j\dots\ga'_2 \ga'_1 \int_\eps^1 \ga_{j+1}\dots \ga_\ell
\end{align*}
where $\ga'_k=\tx_{\xi/\gl(x)}$ if $\ga_k=\tx_{\xi}$ where $\xi=0$ or $\xi^8=1$. 
Note that $\xi/\gl(x)$ is still algebraic.
By the usual regularization procedure we see that the last expression can be written as a polynomial $P(\log(\eps))$
plus $O(\eps \log^\ell(\eps))$, such that all the coefficients of $P$ are $\Q$-linear  combination of the multiple polylogarithms
evaluated at algebraic points. Here we have used the fact that $\log(\gl(x))=\tfrac12\big( \Li_1(-x)-\Li_1(x)\big)$.
Finally, taking $\eps\to 0$ yields the corollary at once.
\end{proof}

\begin{thm}\label{thm-binnSquare}
Keep notation as in Thm.~\ref{thm-Mixed-Parity-CMZVs}. Assume $s_1\ge 3$.

\begin{enumerate}[label=\upshape{(\alph*)},leftmargin=1cm]
 \item \label{enu:thm-binnSquare-2n2n+1}
 Let $l_1(n),\dots,l_d(n)=2n,2n\pm 1$. If $l_1(n)\ne 2n-1$ and
 for all $l_j(n)=2n-1$ ($j\ge2$) we have $l_{j-1}(n)\ne 2n$, then
\begin{align}\label{equ-cor-Mixed-Parity1}
\sum_{n_1 \ \underset{1}{\succ} \ \cdots \ \underset{d-1}{\succ} n_d\ \underset{d}{\succ} \ 0}
 \frac{ b_{n_1}^2}{l_1(n_1)^{s_1}\cdots l_d(n_d)^{s_d}} \in \CMZV_{|\bfs|}^4.
\end{align}
 In particular, if $l_{j}(n)\ne 2n$ for all $j$ then \eqref{equ-cor-Mixed-Parity1} holds. If $l_1(n)=2n-1$ and
 for all $l_j(n)=2n-1$ ($j\ge2$) we have $l_{j-1}(n)\ne 2n$, then
\begin{align*}
\sum_{n_1 \ \underset{1}{\succ} \ \cdots\ \underset{d-1}{\succ} n_d\ \underset{d}{\succ} \ 0}
 \frac{ b_{n_1}^2}{l_1(n_1)^{s_1}\cdots l_d(n_d)^{s_d}} \in \CMZV_{|\bfs|}^4+\CMZV_{|\bfs|+1}^4+\CMZV_{|\bfs|+2}^4.
\end{align*}

 \item \label{enu:thm-binnSquare-gen}
 More generally, if $l_1(n),\dots,l_d(n)=2n,2n\pm 1$ then we have
\begin{align}\label{equ-cor-Mixed-Parity2}
\sum_{n_1\ \underset{1}{\succ}\ \cdots\ \underset{d-1}{\succ} n_d\ \underset{d}{\succ} \ 0}
 \frac{ b_{n_1}^2}{l_1(n_1)^{s_1}\cdots l_d(n_d)^{s_d}} \in \CMZV_{\le |\bfs|+2\nu(l_1)}^4\otimes\Q[i].
\end{align}

 \item \label{enu:thm-binnSquare-Any}
 Moreover, the claim in \ref{enu:thm-binnSquare-gen} still holds if one changes any of the strict inequalities
$n_j>n_{j+1}$ to $n_j\ge n_{j+1}$ in \eqref{equ-cor-Mixed-Parity2} and vice versa,
provided the series is defined. In particular,
\begin{align}\label{eqn:thm-binnSquare-Any}
\sum_{n_1\ \succ\ \cdots\ \succ \ n_d \succ \ 0}
 \frac{ b_{n_1}^2}{l_1(n_1)^{s_1}\cdots l_d(n_d)^{s_d}} \in \CMZV_{\le |\bfs|+2\nu(l_1)}^4\otimes\Q[i],
\end{align}
where ``$\succ$'' can be either ``$\ge$'' or ``$>$'', provided the series is defined.
\end{enumerate}
\end{thm}

\begin{proof} The key observation is that
\begin{align*}
 \int_0^1 \frac{x^{2n+1}}{\sqrt{1-x^2}} \,dx= \int_0^{\pi/2} \sin^{2n+1} t \,dt= B \Big(n+1,\frac12\Big)=\frac{b_n}{2n+1}.
\end{align*}

\medskip
\ref{enu:thm-binnSquare-2n2n+1} When $l_1(n)=2n+1$ by \eqref{eqn-Mixed-Parity-NonTrig} we see that the sum
\begin{align*}
& \sum_{n_1 \ \underset{1}{\succ} \ \dots\underset{d-1}{\succ} n_d\ \underset{d}{\succ} \ 0}
\frac{b_{n_1}(x)/\sqrt{1-x^2} }{ l_1(n_1)^{s_1}\cdots l_d(n_d)^{s_d}}
 = \frac{1}{x\sqrt{1-x^2}} \int_0^x \om_0^{s_1-2} \om_3 \circ \gL_{l_2,s_2}(t) \circ \cdots \gL_{l_d,s_{d}}(t)\circ \om_1
\end{align*}
when $s_1\ge 2$. Thus multiplying by $x$ on both sides and integrating over $(0,1)$ we get
\begin{align*}
 \sum_{n_1 \ \underset{1}{\succ} \ \dots\underset{d-1}{\succ} n_d\ \underset{d}{\succ} \ 0}
 \frac{b_{n_1}^2}{l_1(n_1)^{s_1+1}\cdots l_d(n_d)^{s_d}}
 = &\, \int_0^1 \om_1 \om_0^{s_1-2} \om_3 \circ \gL_{l_2,s_2}(t) \circ \cdots \gL_{l_d,s_{d}}(t)\circ \om_1.
\end{align*}
The claim follows immediately since there are even number of $\om_1$'s in this case.

\medskip
\ref{enu:thm-binnSquare-gen} If $l_1(n)=2n$ then we see that
\begin{align*}
 \sum_{n_1 \ \underset{1}{\succ} \ \dots\underset{d-1}{\succ} n_d\ \underset{d}{\succ} \ 0}
 \frac{b_{n_1}(x)}{l_1(n_1)^{s_1}\cdots l_d(n_d)^{s_d}}
 = &\, \int_0^1 \om_0^{s-2} \om_1 \circ \gL_{l_2,s_2}(t) \circ \cdots \gL_{l_d,s_{d}}(t)\circ \om_1
\end{align*}
when $s_1\ge 2$. Then we can divide by $x\sqrt{1-x^2}$ and integrate over $(0,1)$ to get
\begin{align*}
\sum_{n_1 \ \underset{1}{\succ} \ \dots\underset{d-1}{\succ} n_d\ \underset{d}{\succ} \ 0}
 \frac{b_{n_1}^2}{l_1(n_1)^{s_1+1}\cdots l_d(n_d)^{s_d}}
 =&\, \sum_{n_1 \ \underset{1}{\succ} \ \dots\underset{d-1}{\succ} n_d\ \underset{d}{\succ} \ 0}
 \frac{b_{n_1} b_{n_1-1}}{(2n_1-1) l_1(n_1)^{s_1}\cdots l_d(n_d)^{s_d}} \\
 =&\,\int_0^1 \om_3 \om_0^{s_1-2} \om_1 \circ \gL_{l_2,s_2}(t) \circ \cdots \gL_{l_d,s_{d}}(t)\circ \om_1
\end{align*}
since
\begin{equation}\label{bn-1Tobn}
 \frac{b_{n_1-1}}{2n_1-1}= \frac{b_{n_1}}{2n_1}.
\end{equation}
The theorem holds as well in this case as the number of $\om_1$'s is still even.

\medskip
\ref{enu:thm-binnSquare-Any} If $l_1(n)=2n-1$ then since $s_1\ge 3$ by \eqref{eqn-Mixed-Parity-NonTrig} we have
\begin{align*}
& \sum_{n_1 \ \underset{1}{\succ} \ \dots\underset{d-1}{\succ} n_d\ \underset{d}{\succ} \ n}
\frac{b_{n_1}(x)}{ l_1(n_1)^{s_1}\cdots l_d(n_d)^{s_d}}
 = x\int_0^x (\om_0+1) \om_0^{s-2} \om_3\circ\gL_{l_2,s_2}(t) \circ \cdots \gL_{l_d,s_{d}}(t)\circ b_{n}(t) \om_1.
\end{align*}
We first differentiate this to get
\begin{align*}
\sum_{n_1 \ \underset{1}{\succ} \ \dots\underset{d-1}{\succ} n_d\ \underset{d}{\succ} \ n}
\frac{2n_1 b_{n_1} x^{2n_1-1} }{ l_1(n_1)^{s_1}\cdots l_d(n_d)^{s_d}}
 = &\,\int_0^x (\om_0+1)^2 \om_0^{s-3} \om_3\circ\gL_{l_2,s_2}(t) \circ \cdots \gL_{l_d,s_{d}}(t)\circ b_{n}(t) \om_1.
\end{align*}
As in the $l_1(n)=2n$ case, we can divide by $\sqrt{1-x^2}$ and integrate over $(0,1)$ to get
\begin{align*}
& \sum_{n_1 \ \underset{1}{\succ} \ \dots\underset{d-1}{\succ} n_d\ \underset{d}{\succ} \ n}
\frac{b_{n_1}^2}{ l_1(n_1)^{s_1}\cdots l_d(n_d)^{s_d}}
 = \int_0^1 \om_1(\om_0+1)^2 \om_0^{s_1-3} \om_3\circ\gL_{l_2,s_2}(t) \circ \cdots \gL_{l_d,s_{d}}(t)\circ b_{n}(t) \om_1
\end{align*}
by using \eqref{bn-1Tobn} again. This completes the proof of the theorem.
\end{proof}

\section{A corollary and some examples}
In this last section, we will apply our main theorems to compute a few typical Ap\'ery type series to illustrate the power of our method. We can also see how the regularization process is needed in some of the examples.

First, we can answer affirmatively a few questions we posted at the end of \cite{XuZhao2021b}.
For $\bfk\in\N^d$ and $\bfl\in\N^e$ we define
\begin{align*}
\ze_n(\bfk):=&\, \sum_{n\ge m_1>\dots>m_d>0} \frac{1}{m_1^{k_1}\cdots m_d^{k_d}}, \\
t_n(\bfl):=&\, \sum_{n\ge r_1>\dots>r_e>0} \frac{1}{(2r_1-1)^{l_1}\cdots (2r_e-1)^{l_e}}.
\end{align*}

\begin{cor}\label{cor-answerQuestions}
For all $m\in\N$, $p\in\N_{\ge2}$, $q\in\N_{\ge3}$,
and all compositions of positive integers $\bfk$ and $\bfl$ (including the cases $\bfk=\emptyset$ or $\bfl=\emptyset$),
we have
\begin{alignat*}{4}
{\rm (a)}& \ \ \sum_{n=1}^\infty b_n \frac{\ze_n(\bfk)t_n(\bfl)}{n^{p}} \in \CMZV_{|\bfk|+|\bfl|+p}^4,& \quad
{\rm (b)}& \ \ \sum_{n=1}^\infty b_n^2 \frac{\ze_n(\bfk)t_n(\bfl)}{n^{q}} \in \CMZV^{4}_{|\bfk|+|\bfl|+q},\\
{\rm (c)}& \ \ \sum_{n=0}^\infty b_n \frac{\ze_n(\bfk)t_n(\bfl)}{(2n+1)^{p}} \in i\CMZV^{4}_{|\bfk|+|\bfl|+p},& \quad
{\rm (d)}& \ \ \sum_{n=0}^\infty b_n^2 \frac{\ze_n(\bfk)t_n(\bfl)}{(2n+1)^{q}} \in \CMZV^{4}_{|\bfk|+|\bfl|+q}.
\end{alignat*}
\end{cor}
\begin{proof} Write
\begin{align*}
\ze_n(\bfk)=\sum_{n\ge m_1>\dots>m_d>0} \frac{1}{m_1^{k_1}\cdots m_d^{k_d}},
\quad
t_n(\bfl)= \sum_{n> r_1>\dots>r_e \ge 0} \frac{1}{(2r_1+1)^{l_1}\cdots (2r_e+1)^{l_e}}.
\end{align*}
We only need to note
the following facts: (i) for any summation index $m$ for $\ze_n(\bfk)$ and summation index $r$ for $t_n(\bfl)$ there
are only two possibilities: $m>r$ or $r\ge m$; (ii) we can re-write
\begin{align*}
 \sum_{n>r_1} \frac{1}{(2n+1)^q (2r_1+1)^{l_1}}= \sum_{n\ge r_1} \frac{1}{(2n+1)^q (2r_1+1)^{l_1}}-\frac{1}{(2n+1)^{q+l_1}}
\end{align*}
and obtain similar identities when $n$ and $r_1$ are replaced by $r_{j}$ and $r_{j+1}$. Therefore, we see that
(a) and (c) are special cases of Thm.~\ref{thm-Mixed-Parity-CMZVs}\ref{enu:thm-Mixed-Parity-2n2n+1}.
(b) and (d) are special cases of Thm.~\ref{thm-binnSquare}\ref{enu:thm-binnSquare-2n2n+1}.
\end{proof}

In the following we will compute a series of examples using our main theorems.
\begin{ex}
When depth $d=1$, by \eqref{bn-it4} we see that for all $x\in[-1,1]$
\begin{align}\label{Formula-odd-x}
 \sum_{n\ge 0} \frac{b_n(x)}{(2n+1)^{m+2}}
 =&\, \frac1{x}\int_0^x \om_0^m \om_3 \om_1
\end{align}
for all $m\ge 0$. Applying $t\to \frac{1-t^2}{1+t^2}$ to \eqref{Formula-odd-x} we have
\begin{align*}
 \sum_{n\ge 0} \frac{b_n(x)}{(2n+1)^{m+2}}
 =&\,  \frac{i (-1)^m}{x} \int_{\gl(x)}^1 (\tx_{i}-\tx_{-i})(\tx_1-\tx_{-1}) \ty^m,
\end{align*}
where $\gl(x)=\sqrt{\frac{1-x}{1+x}}$ and $\ty=\tx_{-i}+\tx_{i}-\tx_{-1}-\tx_{1}$ as defined in \eqref{equ:changeVar1}. 
Taking $x=1$ and applying Au's Mathematica package \cite{Au2020} we get
\begin{align}\label{eq-Catalan}
 \sum_{n\ge 0} \frac{b_n}{(2n+1)^{2}}
 =&\,2\, \Im (\Li_{1,1}(i,-i)+\Li_{1,1}(-i,-i))=2G\approx 1.83193119, \\
 \sum_{n\ge 0} \frac{b_n}{(2n+1)^{3}}
 =&\,2\, \Im \Big(\Li_{1_3}(-i, i, i) + \Li_{1_3}(-i, i, -i) - \Li_{1_3}(-i, i, -1) - \Li_{1_3}(-i, i, 1) \nonumber\\
 -& \, \Li_{1_3}(-i, -i, -i) - \Li_{1_3}(-i, -i, i) + \Li_{1_3}(-i, -i, -1) + \Li_{1_3}(-i, -i, 1) \Big)\nonumber\\
 =&-\frac{\pi^3}{32}-\frac1{8}\pi\log^2 2+4{\rm Im}\Li_3\left(\frac{1+i}{2}\right)\approx 1.122690025,\nonumber
\end{align}
where $G=\gb(2)$ is Catalan's constant. This sum appears in \cite[Example 2.12]{CampbellCA2022}, too.
\end{ex}

\begin{ex}
As an application of Cor.~\ref{cor:algPts}, we now compute 
\begin{align*}
 \sum_{n\ge 0} \frac{1}{\binn(2n+1)^2}= \sum_{n\ge 0} \frac{b_n(1/2)}{(2n+1)^2}.
\end{align*}
From the previous example we see that 
\begin{align*}
 &\,\sum_{n\ge 0} \frac{1}{\binn(2n+1)^2}= 
 = \int_{\tfrac{1}{\sqrt{3}}}^1 (\tx_{i}-\tx_{-i})(\tx_1-\tx_{-1})  \\
 =&\,2i\left( \int_{\tfrac{1}{\sqrt{3}}}^0 (\tx_{i}-\tx_{-i})(\tx_1-\tx_{-1}) 
 +\int_0^1 (\tx_{i}-\tx_{-i})\int_{1/\sqrt{3}}^0 (\tx_1-\tx_{-1})  
 +\int_0^1 (\tx_{i}-\tx_{-i})(\tx_1-\tx_{-1}) \right) \\
 =&\, 2i\left(\int_0^{\tfrac{1}{\sqrt{3}}} (\tx_1-\tx_{-1}) (\tx_{i}-\tx_{-i})
-\int_0^1 (\tx_{i}-\tx_{-i})\int_0^{\tfrac{1}{\sqrt{3}}}  (\tx_1-\tx_{-1})
\right)+4G  \quad (\text{by \eqref{eq-Catalan}})\\
 =&\, 4\Im\Big(\Li_{1,1}\Big(\frac{-1}{\sqrt{3}},i\Big)-\Li_{1,1}\Big(\frac{1}{\sqrt{3}},-i\Big) \Big)
-\pi\log(2+\sqrt{3}) +4G 
 \approx 1.063459833.
\end{align*} 
\end{ex}

\begin{ex}
As an easy example of Thm.~\ref{thm-Mixed-Parity-CMZVs}\ref{enu:thm-Mixed-Parity-2n2n+1},
by \eqref{bn-it2} and \eqref{bn-it3} we have
\begin{align*}
 \sum_{n_1>n_2\ge 0} \frac{ b_{n_1}}{n_1^2 (2n_2+1)}
=&\, 4\int_0^{\pi/2} dt \circ (\csc t\sec t dt) \circ dt\\
=&\, 4\int_0^1 \om_1 \circ \om_{20} \circ\om_1 \qquad (\text{by $t\to \sin^{-1} t$})\\
=&\,- 4\int_0^1 (\tx_{-i}-\tx_{i})\circ (\tx_{0}+\tx_{-1}+\tx_{1})\circ (\tx_{-i}-\tx_{i})=7\ze(3)
\end{align*}
by the change of variables $t\to (1-t^2)/(1+t^2)$ then using Au's package \cite{Au2020}.
\end{ex}

\begin{ex}
For a pure $\chi$-sum, by \eqref{bn-it6} we have
\begin{align*}
 \sum_{n>0} \frac{b_n}{(2n-1)^2} = &\, \int_0^1 (\om_0+1) \om_3 \om_1
 =i\int_0^1 \td_{-i,i} \td_{-1,1} (1-\ty) \\
 =&\, 2 G-\frac1{32} \pi^3 + 4 \Im\Li_3\Big(\frac{1+i}2\Big) - \frac18 \pi\log^2 2\approx 2.954621213,
\end{align*}
where we see the weight can increase by one as predicted by
Thm.~\ref{thm-Mixed-Parity-CMZVs}\ref{enu:thm-Mixed-Parity-2n-12n+1}. Similarly,
\begin{align*}
 \sum_{n>0} \frac{b_n}{(2n-1)^3} = &\, \int_0^1 (\om_0+1) \om_0 \om_3 \om_1
 =i\int_0^1 \td_{-i,i} \td_{-1,1} \ty(\ty-1) \\
=&\, -4\gb(4)+\frac1{96}\bigg(2\Im\Li_4\Big(\frac{1+i}2\Big)
+4\pi\log^3 2 +3\pi^3 \log 2\\
&\,\hskip3cm -12\pi\log^2 2-3\pi^3 -\Im\Li_3\Big(\frac{1+i}2\Big) \bigg)\approx 2.1543060048.
\end{align*}
\end{ex}

\begin{ex}
For a sum of mixed parities as examples of
Thm.~\ref{thm-Mixed-Parity-CMZVs}\ref{enu:thm-Mixed-Parity-2n-12n+1},
by \eqref{bn-it3} and \eqref{bn-it6} we have
\begin{align*}
 \sum_{n_1>n_2\ge0 } \frac{b_{n_1}}{(2n_1-1)^2(2n_2+1)} = &\, \int_0^1 (\om_0+1) \om_3\om_{20} \om_1
 =i\int_0^1 \td_{-i,i}(\ta-\tx_{-1}-\tx_{1}) \td_{-1,1} (\ty-1) \\
 =&\, 14\gb(4) - 16 \Im\Li_4\Big(\frac{1+i}2\Big)
 - \frac1{12}\pi\log^3 2 - \frac3{16}\pi^3 \log 2 \\
&\,\hskip2cm +\frac1{8}\pi\log^2 2 + \frac5{32}\pi^3 - 4\Im\Li_3\Big(\frac{1+i}2\Big) \approx 3.937040753.
\end{align*}
So we see the weight can increase by one with a starting $\chi$-block.
\end{ex}

\begin{ex}
For a sum of mixed parities without $\gs$-block but with a starting $\gt^\star$-block,
by \eqref{bn-it4} and \eqref{bn-it5} we have
\begin{align*}
 &\, \sum_{n_1\ge n_2>0} \frac{b_{n_1}}{(2n_1+1)^2(2n_2-1)}
 = \int_0^1 \om_3(\om_5\om_3 +\om_2)\om_1\\
 = &\, \int_0^1 \om_3\om_{20} \om_1-\int_0^1 \om_0 \om_3 \om_1 \qquad(\text{since } \om_5=-d\sqrt{1-t^2}, \ \sqrt{1-t^2}\om_3=\om_0)\\
=&\, -i\int_0^1 \td_{-i,i}(\ty+\tz) \td_{-1,1}+i\int_0^1 \td_{-i,i}\td_{-1,1} \ty\\
 =&\, \frac3{16}\pi^3- 8\Im\Li_3\Big(\frac{1+i}2\Big) +\frac14\pi\log^2 2 \approx 1.630404535576,
\end{align*}
where $\ty+\tz=-\ta-\tx_{-1}-\tx_{1}$. Thus the weight is unchanged as predicted by
Thm.~\ref{thm-Mixed-Parity-CMZVs}\ref{enu:thm-Mixed-Parity-2n-12n+1}.
\end{ex}

\begin{ex} For a sum of mixed parities with a starting $\gs$-block,
followed by a $\gt$-block then a $\chi$-block, by \eqref{bn-it2}, \eqref{bn-it3} and \eqref{bn-it6} we get
\begin{align*}
 &\, \sum_{n_1>n_2\ge n_3>0} \frac{b_{n_1}}{(2n_1)^2(2n_2+1)(2n_3-1)^2}
 = \int_0^1 \om_1\om_{20}\om_5(\om_1+1)\om_3 \om_1\\
 = &\, \int_0^1 \om_1 \big(\om_{20}\om_3-\om_3\om_0\big)\om_3\om_1 \qquad(\text{since } \om_5=-d\sqrt{1-t^2}, \ \sqrt{1-t^2}\om_3=\om_0)\\
= &\,\int_0^1 \td_{-i,i} \td_{-1,1} \td_{-1,1} (\ty+\tz)\td_{-i,i} -\int_0^1 \td_{-i,i}\td_{-1,1} \ty\td_{-1,1}\td_{-i,i}\\
 =&\,\frac{G}{4}\Big(\frac{\pi^3}{4}-32\Im\Li_3\Big(\frac{1+i}2\Big)+
\pi\log^22\Big)
-\frac{15}2\bigg(\Li_5\Big(\frac12\Big)+\log2\Li_4\Big(\frac12\Big)\bigg)-\frac14\log^52\\
&\,+6\pi\gb(4)+
24\Re\Li_{3,1,1}(1,1,I)+\frac1{384}\bigg(80\pi^2\log^32-15\pi^4\log2-87\pi^2\ze(3)-2250\ze(5)\bigg)\\
&\,\approx 0.98658158829.
\end{align*}
So weight is unchanged in every step as predicted by
Thm.~\ref{thm-Mixed-Parity-CMZVs}\ref{enu:thm-Mixed-Parity-2n-12n+1}.
\end{ex}

\begin{ex}
For a sum of mixed parities with $\gs$-block followed by a $\chi$-block,
by \eqref{bn-it2} and \eqref{bn-it5} we obtain
\begin{align*}
 &\, \sum_{n_1>n_2>0} \frac{b_{n_1}}{(2n_1)^2(2n_2-1)}\\
=&\, \int_0^{\pi/2} dt\,dt \Big(\sin t\,dt \csc t\,dt +\tan t\,dt\Big) \,dt \\
=&\, \int_0^{\pi/2} dt\,dt \Big(d(-\cos t)\, \csc t\,dt +\tan t\,dt\Big)\,dt \\
=&\, \int_0^{\pi/2} dt\, \cos t\, dt \csc t\,dt+ dt\Big(\cot t\,dt +\tan t\,dt\Big) \,dt \\
=&\, \int_0^{\pi/2} dt\, (\sin t-1) \csc t\, dt+ dt \,\csc t\sec t\,dt\,dt \\
=&\, \int_0^{\pi/2} dt\,dt- \csc t\, dt+ dt \,\csc t\sec t\,dt \,dt \\
= &\, \int_0^1 \om_1\om_1-\om_3\om_1 +\om_1\om_{20}\om_1 \\
= &\, -i \int_0^1 \td_{-i,i} \td_{-1,1}+\int_0^1 \td_{-i,i}(\ty+tz)\td_{-i,i}-\td_{-i,i}^2\\
=&\,\frac18\pi^2 -2G+\frac74\zeta(3) \approx 1.5053689423.
\end{align*}
Note that not only the weight is a mix of 2 and 3, but this is a mix of both real
and imaginary parts of some CMZVs of level 4. The main complication is brought in by the 1-form $\sin t\,dt$
(corresponding to $\om_5$)
appearing in the $\chi$-block, which moves to the front after integration by parts if the block in front is a $\gs$-block
but disappears if the block in front is a $\gt^\star$-block.
\end{ex}

\begin{ex} \label{exampleS}
We apply the idea of proof of Thm.~\ref{thm-Mixed-Parity} to the sum
\begin{equation*}
S:=\sum_{n_1\ge n_2\ge n_3\ge 1} \frac{b_{n_1}}{n_1^2(2n_2+1)(2n_3-1)}.
\end{equation*}
First, we break the sum in two sub-sums $S=S_1+S_2$ where
\begin{align*}
S_1=&\, \sum_{n_1> n_2\ge n_3\ge 1} \frac{b_{n_1}}{n_1^2(2n_2+1)(2n_3-1)},\\
S_2=&\, \sum_{n_1\ge n_3\ge 1} \frac{b_{n_1}}{n_1^2(2n_1+1)(2n_3-1)}.
\end{align*}
Then by \eqref{bn-it2}, \eqref{bn-it3} and \eqref{bn-it5} we see that
\begin{align*}
S_1=&\, 4 \int_0^1 \om_1 \sum_{n_2\ge n_3>0} \frac{b_{n_2}(t)}{(2n_2+1)(2n_3-1)} \om_1\\
=&\, 4 \int_0^1 \om_1 \om_{20}\sum_{n_3>0} \frac{b_{n_3}(t)}{2n_3-1} \om_1\\
=&\, 4 \int_0^1 \om_1 (\om_0+\om_2) \om_5 \om_3 \om_1
+4 \int_0^1 \om_1 (\om_0+\om_2) \om_2 \om_1.
\end{align*}
By the change of variables $t\to \frac{1-t^2}{1+t^2}$ we obtain
\begin{align*}
S_1=&\,-4 A-4 \bar{A} -4 \int_0^1 \td_{-i,i}(\ta+\tx_{-i}+\tx_i)(\ta+\tx_{-1}+\tx_1) \td_{-i,i},
\end{align*}
where $\bar{A}$ is the complex conjugation of
\begin{align*}
A= \int_0^1 \td_{-i,i} \td_{-1,1} \frac{dt}{(i-t)^2} (\ta+\tx_{-1}+\tx_1) \td_{-i,i}.
\end{align*}
Integration by parts yields
\begin{align*}
A= &\, \int_0^1 \td_{-i,i}\frac{\tx_{-1}-\tx_1}{i-t} (\ta+\tx_{-1}+\tx_1) \td_{-i,i}
-\int_0^1 \td_{-i,i} \td_{-1,1} \frac{\ta+\tx_{-1}+\tx_1}{i-t} \td_{-i,i}.
\end{align*}
Explicitly, for all fourth roots of unity $\xi\ne i$
\begin{align}
\frac{\ta}{i-t}=i(\tx_i-\ta), \
\frac{\tx_\xi}{i-t}=\frac{\td_{i,\xi}}{\xi-i}.
\end{align}
We obtain
\begin{align*}
A= &\, \int_0^1 \td_{-i,i}\left(\frac1{-1-i} \td_{i,-1}-\frac1{1-i} \td_{i,1}\right) (\ta+\tx_{-1}+\tx_1) \td_{-i,i} \\
-&\, \int_0^1 \td_{-i,i} \td_{-1,1} \left(-i(\tx_i+\ta)-\frac1{1+i} \td_{i,-1}+\frac1{1-i} \td_{i,1}\right) \td_{-i,i}\\
= &\, \int_0^1 \td_{-i,i}\left(-\tx_i+\frac{1-i}{2}\tx_{-1}+\frac{1+i}{2}\tx_{1} \right) (\ta+\tx_{-1}+\tx_1) \td_{-i,i} \\
-&\, \int_0^1 \td_{-i,i} \td_{-1,1} \left(-i\ta+\frac{1-i}{2}\tx_{-1}-\frac{1+i}{2}\tx_{1}\right) \td_{-i,i}.
\end{align*}
Therefore using Au's package \cite{Au2020} we find that
\begin{align*}
S_1=-4 \int_0^1 \Big( \td_{-i,i}(\ta+\tx_{-1}+\tx_1)^2 \td_{-i,i}
- \td_{-i,i} \td_{-1,1}^2 \td_{-i,i}\Big)=8G^2\approx 6.71194375752575.
\end{align*}

Now we turn to $S_2$. Set \label{S2} 
\begin{align*}
S_2(x)=&\,\sum_{n_1\ge n_3\ge 1}
 \frac{b_{n_1}(x)}{n_1^2(2n_1+1)(2n_3-1)}.
\end{align*}
By partial fraction
\begin{align*}
S_2(x)=&\, \sum_{m\ge n>0} \left(\frac{4 b_{m}(x)}{(2m+1)(2n-1)}-\frac{2 b_{m}(x)}{m(2n-1)}
+ \frac{b_{m}(x)}{m^2(2n-1)}\right)\\
=&\, \sum_{m\ge n>0} \frac{4 b_{m}(x)}{(2m+1)(2n-1)}- \sum_{m>n>0} \left(\frac{2 b_{m}(x)}{m(2n-1)}
-\frac{b_{m}(x)}{m^2(2n-1)}\right)-\sum_{n>0} \left(\frac{2 b_{n}(x)}{n(2n-1)}
-\frac{b_{n}(x)}{n^2(2n-1)}\right)\\
=&\, 4\sum_{n>0} \left( f_{20}(x)\int_0^x\frac{b_{n}(t)}{2n-1} \om_1
-f_2(x)\int_0^x \frac{b_{n}(t)}{2n-1} \om_1
+\int_0^x \om_1 \frac{b_{n}(t)}{2n-1} \om_1
-\frac{b_{n}(x)}{4 n^2} \right) \\
=&\, 4 \big(f_{20}(x)-f_2(x)\big)\int_0^x ( \om_5\om_3 \om_1+\om_2 \om_1)
+4\int_0^x \om_1 ( \om_5\om_3 \om_1+\om_2 \om_1)-4\int_0^x \om_1 \om_1
\end{align*}
by \eqref{bn-it1}--\eqref{bn-it5}. Note that
\begin{align*}
\lim_{x\to 1^-} \int_0^1 &\, \om_5\om_3 \om_1= -i\int_0^1 \td_{-i,i}\td_{-1,1}\bigg(\frac{dt}{(i-t)^2}+\frac{dt}{(i+t)^2}\bigg)
= 2\, \Re \left(-i\int_0^1 \td_{-i,i}\frac{\td_{-1,1}}{i-t} \right) \\
=&\, 2\,\Im \int_0^1 \td_{-i,i}\left(\frac1{-1-i} \td_{i,-1}-\frac1{1-i} \td_{i,1}\right)
=\Im \int_0^1 \td_{-i,i}\left(-2\tx_i+(1-i)\tx_{-1}+(1+i)\tx_{1}\right),
\end{align*}
which is a finite value in $i\CMZV_2^4$.
Further, setting $\gl(x)=\sqrt{\frac{1-x}{1+x}}$ we get
\begin{align*}
\int_0^x \om_2 \om_1=&\, -i \int_{\gl(x)}^1 \td_{-i,i} (\ta+\tx_{-i}+\tx_{i}).
\end{align*}
We only need to take care of
\begin{align*}
 \int_{\gl(x)}^1 \tx_{-i} \ta=&\, \int_{\gl(x)}^1 \tx_{-i} \int_{\gl(x)}^1 \ta - \int_{\gl(x)}^1 \ta \tx_{-i}
= -\log\bigg(\frac{i+\gl(x)}{i+1}\bigg) \log\gl(x) - \int_{\gl(x)}^1 \ta \tx_{-i}
\end{align*}
of which the last term $\to -\Li_2(i)$ as $x\to 1^-$. Hence
\begin{align*}
\lim_{x\to 1^-} (f_{20}(x)-f_{2}(x)) \int_{\gl(x)}^1 \tx_{-i} \ta
=&\, -\log\bigg(\frac{i}{i+1}\bigg) \lim_{x\to 1^-} \sqrt{1-x^2}\log\sqrt{\frac{1-x}{1+x}}\\
=&\, -\log\bigg(\frac{i}{i+1}\bigg) \frac{\sqrt{2}}2 \lim_{\eps\to 0^+} \sqrt{\eps} \log(\eps/2) =0.
\end{align*}
Thus
\begin{align*}
S_2=& \, \lim_{x\to 1^-} S_2(x)=4\int_0^1 \om_1 ( \om_5\om_3 \om_1+\om_2 \om_1)-4 \int_0^1 \om_1\om_1 \\
=&\, 4\int_0^1 \td_{-i,i}\bigg(\tz-
\td_{-1,1}\Big(\frac{dt}{(i-t)^2}+\frac{dt}{(i+t)^2}\Big)\bigg)\td_{-i,i}
-2\left(\int_0^1 \om_1\right)^2\\
=&\, -4\int_0^1 \td_{-i,i}(\ta+\tx_{-i}+\tx_{i})\td_{-i,i} -2B-2\bar{B}-\frac{\pi^2}{2},
\end{align*}
where
\begin{align*}
B=&\, 2\int_0^1 \td_{-i,i}\frac{\td_{-1,1}}{i-t} \td_{-i,i}
- 2\int_0^1 \td_{-i,i} \td_{-1,1} \frac{\td_{-i,i}}{i-t} \\
=&\, 2\int_0^1 \td_{-i,i}\left(\frac1{-1-i} \td_{i,-1}-\frac1{1-i} \td_{i,1}\right) \td_{-i,i}
- \,2\int_0^1 \td_{-i,i} \td_{-1,1}\left(\frac1{-2i} (\tx_i-\tx_{-i})-
\frac{1}{i-t}-i \right)\\
=&\, \int_0^1 \td_{-i,i}\big(\tx_{-1}+\tx_{1}-i\td_{-1,1}-2\tx_i\big) \td_{-i,i}
- i\int_0^1 \td_{-i,i} \td_{-1,1} \td_{-i,i}
+2\int_0^1 \td_{-i,i} \frac{\td_{-1,1}}{i-t}+2i\int_0^1 \td_{-i,i} \td_{-1,1}\\
=&\, \int_0^1 \td_{-i,i}\big(\tx_{-1}+\tx_{1}-i\td_{-1,1}-2\tx_i\big) ( \td_{-i,i}+1)
- i\int_0^1 \td_{-i,i} \td_{-1,1} \td_{-i,i}+2i\int_0^1 \td_{-i,i} \td_{-1,1}.
\end{align*}
Hence
\begin{align*}
B+\bar{B}=2\left(\int_0^1 \td_{-i,i}\big(\tx_{-1}+\tx_{1}-\tx_{-i}-\tx_i\big) \td_{-i,i}
+\int_0^1 \td_{-i,i}\td_{-i,i} +i\int_0^1 \td_{-i,i} \td_{-1,1}\right).
\end{align*}
We can see that $S_2=7\ze(3)-8G \approx 1.0866735685$ by Au's package \cite{Au2020} and therefore
\begin{equation*}
 \sum_{n_1\ge n_2\ge n_3\ge 1} \frac{b(n_1)}{n_1^2(2n_2+1)(2n_3-1)}
 =7\ze(3)+8G^2-8G\approx 7.79861732643.
\end{equation*}
\end{ex}

In the next three examples, we consider some Ap\'ery-type series which involve the square of the central binomial coefficients.

\begin{ex} \label{eg:W5}
Since $\chi$-block does not appear in this example we see that the weight of the CMZVs
is the same as the weight of the series, as predicted by Thm.~\ref{thm-binnSquare}\ref{enu:thm-binnSquare-2n2n+1}:
\begin{align*}
\sum_{n_1 \ge n_2>0}
\frac{b_{n_1}^2}{ (2n_1+1)^4(2n_2) }
 = &\, \int_0^1 \om_1 \om_0 \om_3 \om_2 \om_1= \int_0^1\td_{-i,i} \tz \td_{-1,1} \ty \td_{-i,i} =W_5 \approx
 0.04433915814,
\end{align*}
where
\begin{align*}
W_5=&\, 5\log2 \Li_4\Big(\frac12\Big) + 21 \Li_5\Big(\frac12\Big)
 + \pi \bigg(16 \Im\Li_4\Big(\frac{1+i}2\Big)-17 \gb(4) +8 \Im\Li_3\Big(\frac{1+i}2\Big) \log2\bigg) \\
&\,+ \frac{379}{2880} \pi^4 \log2 + \frac{1}{30}\log^5 2- 16 \Re \Li_{3, 1, 1}(1, 1, i) -
 \frac{1}{192} \pi^2 \Big(16 \log^3 2 - 29 \ze(3)\Big) - \frac{27}{4} \ze(5).
\end{align*}

\end{ex}

\begin{ex}
The sum next has weight 4 but due to the $\gs$-$\chi$-block chain we need
to use CMZVs of weight of both 3 and 4 to express it:
\begin{align*}
 \sum_{n_1 > n_2>0}
&\, \frac{b_{n_1}^2}{ (2n_1)^3(2n_2-1) }
 =\int_0^1 \om_3 \om_1 (\om_5\om_3+\om_2) \om_1
 = \int_0^1 \om_3 (\om_1\om_0-dt\om_3) \om_1+\om_3 \om_1\om_2\om_1 \\
 &\, \ \hskip2cm \qquad(\text{since } \om_5=-d\sqrt{1-t^2}, \ \sqrt{1-t^2}\om_3=\om_0, \ \sqrt{1-t^2}\om_1=dt)\\
 =&\,\int_0^1 (\om_3 \om_1-\om_1 \om_3)\om_1+\om_3\om_1\om_0\om_1+\om_3 \om_1\om_2\om_1 \qquad(\text{since } t\om_3=\om_1)\\
 =&\,\int_0^1 \td_{-i,i}(\td_{-i,i}\td_{-1,1}-\td_{-1,1}\td_{-i,i})-\td_{-i,i}(\ty+\tz) \td_{-i,i}\td_{-1,1}\\
 =&\, 2 G^2 - G \pi \log2 +
 \frac{1}{64} \pi \bigg(3 \pi^3 - 128 \Im\Li_3\Big(\frac{1+i}2\Big) + 4 \pi \log^2 2\bigg) + 2 G \pi - \frac{21}{4}\ze(3) \\
 \approx &\, 0.40829155182.
\end{align*}
\end{ex}

\begin{ex}
The series in this example does not have a $\gs$-$\chi$-block chain so that there is no weight drop. But the first block is a
$\chi$-block so the weight can increase by two as predicted by Thm.~\ref{thm-binnSquare}\ref{enu:thm-binnSquare-gen}:
\begin{align*}
\sum_{n_1 > n_2>0} \frac{b_{n_1}^2}{ (2n_1-1)^3(2n_2) }
 =& \, \int_0^1 \om_1(\om_0+1)^2 \om_3 \om_2\om_1 \\
 =&\,-\int_0^1 \td_{-i,i} \tz \td_{-1,1} (\ty-1)^2 \td_{-i,i} =W_6+2W_5+W_4\approx 0.38530528471,
\end{align*}
where $W_5$ is defined in Example~\ref{eg:W5} and
\begin{align*}
W_4=&\, -2 G^2 - \frac{49}{720} \pi^4 + 2 \pi \Im\Li_3\Big(\frac{1+i}2\Big) -
 \frac{11}{48} \pi^2 \log^2 2+ \frac{1}{6}\log^4 2 + G \pi\log^2 2 +
 4 \Li_4\Big(\frac12\Big), \\
W_6=&\, 68 \Li_6\Big(\frac12\Big) -\frac{7655}{27648} \pi^6 +
 \frac{61}{2} \pi \Im \Li_{4, 1}(i, 1) - \frac{41}{2} \pi \Im \Li_{4, 1}(i, -1) + 96 \pi \Im\Li_5\Big(\frac{1+i}2\Big) \\
 &\, - 19 \pi \gb(4) \log2 + 32 \pi \Im\Li_4\Big(\frac{1+i}2\Big) \log2 - \frac{181}{2880} \pi^4 \log^2 2 - \frac{1}{96} \pi^2 \log^4 2 + \frac{1}{90}\log^62 \\
 &\, - \frac{169}{4} \ze(\bar5,1) - \frac{5}{12} \pi^2 \Li_4\Big(\frac12\Big) +
 10 \log2 \Li_5\Big(\frac12\Big) - 24 G \gb(4) - 6 \pi^2 \log2 \ze(3) \\
 &\, -24 \Re \Li_{4, 2}(-1, i) + 64 \Re \Li_{3, 1, 1, 1}(1, 1, 1, i) + \frac{8195}{128} \ze(3)^2 +
 \frac{2821}{32} \log2 \ze(5).
\end{align*}
\end{ex}

\medskip

\noindent{\bf Acknowledgement.} Ce Xu is supported by the National Natural Science Foundation of China [Grant No. 12101008], the Natural Science Foundation of Anhui Province [Grant No. 2108085QA01] and the University Natural Science Research Project of Anhui Province [Grant No. KJ2020A0057]. Jianqiang Zhao is supported by the Jacobs Prize from The Bishop's School.

\medskip



\begin{thebibliography}{99}

\bibitem{Ablinger2015}
J.\ Ablinger, Discovering and proving infinite binomial sums identities,
\emph{Experimental Math.} \textbf{\textbf{26}} (2017), pp.\ 62--71.
arXiv:1507.01703.

\bibitem{Akhilesh1}
P.\ Akhilesh, Double tails of multiple zeta values, \emph{J. Number Thy.} \textbf{170}(2017), pp.\ 228--249.

\bibitem{Akhilesh}
P.\ Akhilesh, Multiple zeta values and multiple Ap\'ery-like sums, \emph{J. Number Thy.} \textbf{226}(2021), pp.\ 72--138.

\bibitem{Au2020}
K.C. Au, Evaluation of one-dimensional polylogarithmic integral, with applications to infinite series, arXiv:2007.03957. A companion Mathematica package available at researchgate.net/publication/342344452

\bibitem{CampbellCA2022}
J.\ M.\ Campbell, M.\ Cantarini and J.\ D'Aurizio,
Symbolic computations via Fourier--Legendre expansions and fractional operators,
\emph{Integral Transforms and Special Func.} \textbf{33}(2)(2022), pp.\ 1--19.

\bibitem{KTChen1971}
K.-T.\ Chen, Algebras of iterated path integrals and fundamental groups, \emph{Trans.\ Amer.\ Math.\ Soc.} \textbf{156}(1971), pp.\ 359--379.

\bibitem{KTChen1977}
K.-T.\ Chen, Iterated path integrals, \emph{Bull.\ Amer.\ Math.\ Soc.} \textbf{83}(1977), pp.\ 831--879.

\bibitem{DavydychevDe2001}
A.I.\ Davydychev and M.\ Yu.\ Kalmykov, New results for the epsilon-expansion of
certain one-, two- and three-loop Feynman diagrams,
\emph{Nucl. Phys. B} \textbf{605} (2001), pp.\ 266--318. arXiv:hep-th/0012189.

\bibitem{DavydychevDe2004}
A.I.\ Davydychev and M.\ Yu.\ Kalmykov, Massive Feynman diagrams and inverse binomial sums, \emph{Nuclear Phys.\ B}
\textbf{699} (2004), pp.\ 3--64. arXiv:hep-th/0303162v4.

\bibitem{JegerlehnerKV2003}
F. Jegerlehner, M.Yu. Kalmykov and O. Veretin, $\overline{\rm{MS}}$ versus pole masses of gauge bosons II:
two-loop electroweak Fermion corrections,
\emph{Nucl. Phys.} \textbf{B658} (2003), pp.\ 49--112.


\bibitem{KWY2007}
M.Yu. Kalmykov, B.F.L. Ward, S.A. Yost,
Multiple (inverse) binomial sums of arbitrary weight and depth and the all-order $\eps$-expansion of generalized hypergeometric functions with one half-integer value of parameter,
\emph{J. High Energy Phys}. \textbf{2007} (10)(2007) 048, 26 pp.

\bibitem{LLO2022}
L. Lai, C. Lupu and D. Orr, Elementary proofs of Zagier's formula for multiple zeta values and its odd variant, arXiv:2201.09262.

\bibitem{LY2020}
L. Lai and P. Yu, A note on the number of irrational odd zeta values, \emph{Compos. Math.} \textbf{156}(2020), no. 8, pp.\ 1699--1717.

\bibitem{Leshchiner}
D.\ Leshchiner, Some new identities for $\gz(k)$, \emph{J. Number Theory}, \textbf{13}(1981), pp.\ 355--362. MR0634205
(83k:10072.

\bibitem{Murakami2021}
T. Murakami, On Hoffman's $t$-values of maximal height and generators
of multiple zeta values, \emph{Math. Ann.}, \textbf{382}(2022), pp.\ 421--458.

\bibitem{Racinet2002}
G.\ Racinet, Doubles m\'elanges des polylogarithmes multiples aux
racines de l'unit\'e (in French),
\emph{Publ.\ Math.\ IHES} \textbf{95} (2002), pp.\ 185--231.


\bibitem{Rivoal2000}
T.\ Rivoal, La fonction z\^eta de Riemann prend une infinit\'e de valeurs
irrationnelles aux entiers impairs (in French),
\emph{C.\ R.\ Acad.\ Sci.\ Ser.\ A. Math.} \textbf{331}(2000), pp.\ 267--270.

\bibitem{Sun2015}
Z.-W.\ Sun, New series for some special values of $L$-functions, \emph{Nanjing Univ. J. Math. Biquarterly} \textbf{32}(2015), no.2, pp.\ 189--218.

\bibitem{X2020}
C. Xu, Explicit relations between multiple zeta values and related variants, \emph{Adv. Appl. Math.} \textbf{130}(2021), 102245.

\bibitem{XuZhao2021b}
C.\ Xu and J.\ Zhao, Ap\'{e}ry-type series and colored multiple zeta values, arXiv:2111.10998.

\bibitem{Zhao2016}
J. Zhao, \emph{Multiple zeta functions, multiple polylogarithms and their special values}, Series on Number
Theory and its Applications, Vol.~12, World Scientific Publishing Co. Pte. Ltd., Hackensack, NJ, 2016.

\bibitem{Zudilin2001}
W.\ Zudilin, One of the numbers $\zeta(5)$, $\zeta(7)$, $\zeta(9)$, $\zeta(11)$ is irrational,
\emph{Russian Math.\ Surveys} \textbf{56}(4)(2001), pp.\ 774--776.

\end{thebibliography}
\end{document}